\documentclass[onefignum,onetabnum]{siamart171218}

\usepackage{tikz}
\usepackage{lipsum}
\usepackage{amsfonts}
\usepackage{graphicx}
\usepackage{epstopdf}
\usepackage{algorithmic}
\usepackage{subcaption}
\ifpdf
  \DeclareGraphicsExtensions{.eps,.pdf,.png,.jpg}
\else
  \DeclareGraphicsExtensions{.eps}
\fi

\newsiamremark{remark}{Remark}
\newsiamremark{hypothesis}{Hypothesis}
\crefname{hypothesis}{Hypothesis}{Hypotheses}
\newsiamthm{claim}{Claim}

\headers{A Recovery-Based Error Indicator for Finite Difference Methods}{F. Sindy, A. Buffa, and M. Picasso}

\title{A Recovery-Based Error Indicator for Finite Difference Methods\thanks{
\funding{This work has received funding via the Euratom Research and Training Programme (Grant Agreement No. 101052200 — EUROfusion) and from the Swiss State Secretariat for Education, Research and Innovation (SERI). 
Views and opinions expressed are those of the author(s) only and do not necessarily reflect those of the European Union, the European Commission, or SERI. Neither the European Union, the European Commission, nor SERI can be held responsible for them.}}}

\author{Ferhat Sindy\thanks{\'Ecole Polytechnique F\'ed\'erale de Lausanne (EPFL), Institute of Mathematics, CH-1015 Lausanne, Switzerland 
  (\email{ferhat.sindy@epfl.ch}, \email{annalisa.buffa@epfl.ch}, \email{marco.picasso@epfl.ch}).} \and Annalisa Buffa\footnotemark[2]\and Marco Picasso\footnotemark[2]}

\usepackage{amsopn}

\makeatletter
\newcommand*{\addFileDependency}[1]{% argument=file name and extension
  \typeout{(#1)}% latexmk will find this if $recorder=0 (however, in that case, it will ignore #1 if it is a .aux or .pdf file etc and it exists! if it doesn't exist, it will appear in the list of dependents regardless)
  \@addtofilelist{#1}% if you want it to appear in \listfiles, not really necessary and latexmk doesn't use this
  \IfFileExists{#1}{}{\typeout{No file #1.}}% latexmk will find this message if #1 doesn't exist (yet)
}
\makeatother

\newcommand*{\myexternaldocument}[1]{%
    \externaldocument{#1}%
    \addFileDependency{#1.tex}%
    \addFileDependency{#1.aux}%
}
\ifpdf
\hypersetup{
  pdftitle={A Recovery-Based Error Indicator for Finite Difference Methods},
  pdfauthor={F. Sindy, A. Buffa, and M. Picasso}
}
\fi

\myexternaldocument{ex_supplement}

\begin{document}

\maketitle

% REQUIRED
\begin{abstract}
A novel recovery-based error indicator for high-order Finite Difference Methods, based on post-processing of the Finite Difference values is presented. The values obtained on the Finite Difference grid are interpolated into a suitable polynomial Finite Element space. A recovery-based error indicator, with the polynomial-preserving property of \cite{Naga2004,Zhang2005}, is then applied to estimate the gradient error. The performance and accuracy of the proposed error indicator are demonstrated through several numerical experiments, including the two-dimensional Poisson problem solved using second- and fourth-order finite difference schemes. Additional experiments are conducted on elliptic problems with discontinuous coefficients, as well as on the two and three-dimensional wave equation in homogeneous media with second- and fourth-order finite differences, and in heterogeneous media with second-order finite differences.
\end{abstract}

% REQUIRED

% REQUIRED

\section{Introduction}
\label{intro}
The theory and practice of a posteriori error estimation are well established for the Finite Element Method (FEM), see for instance  \cite{Ainsworth1997}. In contrast, error estimators for high-order Finite Difference Method (FDM) remain scarce. Our final objective is to derive an a posteriori error indicator suitable for (high-order) FDM discretizations of complex multiphysics problems, for instance, the GBS code for plasma physics in a tokamak \cite{Giacomin2022}.\\
Several works have investigated a posteriori error estimation for FDMs, often by reformulating the finite difference scheme as an equivalent FEM. This reformulation enables the use of standard residual-based estimators \cite{Verfurth2013}. In~\cite{Collins2014}, a FEM formulation is developed that is nodally equivalent to the Lax–Wendroff method. Similarly,~\cite{collins2015posteriori} reformulates explicit FDMs as FEMs for two classes of ODE solvers. In~\cite{mao2025adaptive}, partial integro-differential equations are discretized using IMEXs schemes in time, together with a second-order FDM in space, and continuous piecewise linear reconstructions are employed to obtain FEM-like solutions. In~\cite{chaudhry2017posteriori}, a nodally equivalent FEM is derived for an IMEX scheme. These approaches, however, are limited to low-order discretizations. Also, the main drawback of residual-based error estimators is that, for complex multiphysics problems, their derivation is often nontrivial and must be repeated for each new governing equation. In contrast, recovery-based a posteriori error estimators are equation-agnostic. Their approach relies on post-processing the gradient to obtain a recovered gradient. The gradient error is then estimated as the difference between the recovered gradient and the discrete gradient obtained from the numerical solution.\\
We aim to develop an a posteriori error indicator that is equation-agnostic and therefore well-suited for complex multiphysics problems and high-order FDMs. Examples of such applications include \cite{bohlen2002parallel,taflove2005computational,yee1966numerical,basilisk,HESTHAVEN200359}. Specifically, we propose a novel recovery-based indicator that employs a finite element reconstruction of the numerical solution and its gradient, using the Polynomial Preserving Recovery (PPR) method~\cite{Naga2004,Naga2005,Zhang2005} to estimate the gradient error of the reconstructed solution. The proposed approach offers a flexible and generalizable framework for assessing the accuracy of numerical solutions. While the indicator can also be incorporated into adaptive refinement strategies, the primary focus of this work is on error evaluation and solution assessment. The effectiveness of the proposed indicator is demonstrated on the Poisson problem and the wave equation, though the method could be broadly applicable as a black-box tool for a wide range of PDEs. The paper is organized as follows. In Section \ref{sec:gradient_recovery}, the PPR recovery operator and its properties in the context of FEM are recalled to define a recovery-based error indicator. Section \ref{sec:main} introduces the Recovery-based error indicator for FDM through the poisson problem and wave equation by interpolating the FDM solution in a suitable Finite Element space and using the recovery-based error indicator from Section \ref{sec:gradient_recovery}. Numerical results are provided in Section \ref{sec:experiments}, followed by conclusions in Section \ref{sec:conclusions}.

\section{Recovery-Based Error Indicator}
\label{sec:gradient_recovery}
In this section, we consider the PPR recovery-based error indicator. We closely follow the papers \cite{Naga2005,Naga2004,Zhang2005}.
We let $\Omega\subset\mathbb {R}^d$ be an open, $d$-dimensional polygonal domain. Given $h>0$, let $\mathcal{T}_h$ be a partition of $\Omega$ into elements $K$ with diameter less or equal $h$. We define the continuous piecewise polynomial Finite Element space of degree r: 
$$
\mathcal{V}_h^r=\left\{v_h \in C^0(\overline{\Omega}):\left.v_h\right|_K \in \mathbb{P}_r(K), \quad \forall K \in \mathcal{T}_h\right\}.
$$
\noindent Here $K$ is an interval for $d = 1$, a triangle or quadrilateral for $d = 2$, a tetrahedron or cuboid for $d = 3$, and so on (for quadrilaterals and cuboids we have $\left.v_h\right|_K \in \mathbb{Q}_r(K)$ instead of $\left.v_h\right|_K \in \mathbb{P}_r(K)$). Let $N$ be the number of nodes. A basis of $\mathcal{V}_h^r$ is the standard Lagrange basis functions $\{\phi_i\}_{i = 1}^{N}$ with $\phi_i(z_j) = \delta_{ij}$. For any function $u \in C^0(\overline{\Omega})$, let 
$$
I_h^r u(x) = \sum_{i = 1}^{N}u(z_i)\phi_i(x),\quad x \in \overline{\Omega},
$$
be the Lagrange interpolant of $\mathcal{V}_h^r$. Let $u_h\in \mathcal{V}_h^r$ be an approximation of $u$ obtained, for instance by solving a linear elliptic problem with the FEM. The PPR gradient recovery operator $\mathrm{G}_h$ is a linear mapping from $\mathcal{V}_h^r$ to $\left(\mathcal{V}_h^r\right)^d$, mimicking the $\nabla u$, and defined on piecewise polynomials of degree $r$. Thus: 
$$
\mathrm{G}_hu_h(x) = \sum_{i = 1}^{N}(\mathrm{G}_h u_h)(z_i)\phi_i(x),\quad x \in \overline{\Omega}.
$$
I.e., $\mathrm{G}_h u_h$ is fully defined by its values at the nodes: $\{\mathrm{G}_hu_h(z_i)\}_{i = 1}^{N}$. The process of determining these nodal values consists of three steps:
(1) select local patches of elements; (2) construct local polynomials in a discrete least squares sense; and (3) assemble the recovered data into a global expression. Consider a node $z_i$ associated with a vertex of the mesh, then the local patch is denoted by $\mathcal{K}_i$. It is defined such that the number of nodes around the vertex leads to a unique polynomial $p_{z_i}$. Initially, the local patch is defined as the union of elements that share the vertex. If the number of nodes inside a patch is insufficient, we extend the patch by including mesh elements that share common $(d-1)$-simplices. For a boundary vertex, we extend its local patch by including the local patches associated with the vertices that are inside the boundary vertex's local patch. In Figure \ref{fig:triangle_patch}, we have a triangular mesh for the case $d = 2$ and $r = 1$ on the left and $r = 2$ on the right. We have an internal vertex $z_i$ with its corresponding patch $\mathcal{K}_i$, two boundary vertices $z_k$ and $z_j$ with corresponding patches $\mathcal{K}_k$ and $\mathcal{K}_j$, where the patch $\mathcal{K}_j$ is extended outwards for $r = 1$ since the number of nodes around the vertex is not sufficient if we consider only the elements that share the vertex. In Figure \ref{fig:quad_patch}, we have a quadrilateral mesh showing the same type of vertices and patches for $d = 2$ and $r = 1$ and $r = 2$.
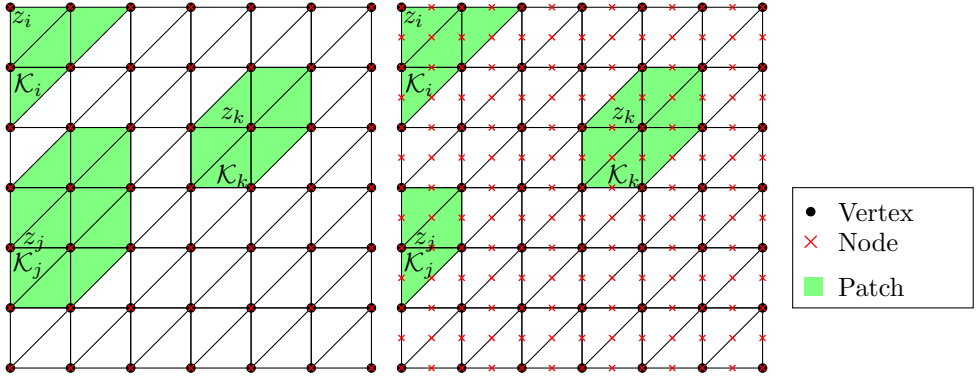
\begin{figure}
    % \begin{tikzpicture}[scale=0.8]

% % Highlight patch 1 (middle)
% \fill[green, opacity=0.5] (3,3) -- (4,3) -- (5,4) -- (5,5) -- (4,5) -- (3,4) -- cycle;

% \fill[green, opacity=0.5] (0,6) -- (2,6) -- (0,4) -- cycle;

% \fill[green, opacity=0.5] (0,1) -- (1,1) -- (2,2) -- (2,4) -- (1,4) -- (0,3) -- cycle;

% % Draw the 6x6 grid with diagonals (structured triangulation)
% \foreach \x in {0,...,5} {
%     \foreach \y in {0,...,5} {
%         % draw square edges
%         \draw[thin] (\x,\y) rectangle (\x+1,\y+1);
%         % add diagonal (bottom-left to top-right)
%         \draw[thin] (\x,\y) -- (\x+1,\y+1);
%     }
% }

% % Draw a circle at each grid vertex
% \foreach \x in {0,...,6} {
%     \foreach \y in {0,...,6} {
%         \filldraw[black] (\x,\y) circle (2pt);
%     }
% }

% \foreach \x in {0,...,6} {
%     \foreach \y in {0,...,6} {
%         \draw[gray, line width=0.5pt] (\x-0.05,\y-0.05) -- (\x+0.05,\y+0.05);
%         \draw[gray, line width=0.5pt] (\x-0.05,\y+0.05) -- (\x+0.05,\y-0.05);
%     }
% }

% % Labels
% \node at (3.7, 4.2) {$z_i$};
% \node at (3.7, 3.2) {$\mathcal{K}_i$};

% \node at (0.4, 2.1) {$z_j$};
% \node at (0.3, 1.7) {$\mathcal{K}_j$};

% \node at (0.4, 5.1) {$z_k$};
% \node at (0.3, 4.7) {$\mathcal{K}_k$};

% \end{tikzpicture}

\begin{tikzpicture}[scale=0.8]

% =====================
% First picture (left)
% =====================
\begin{scope}[shift={(0,0)}]  % <-- left position
  % Highlight patch 1 (middle)
  \fill[green, opacity=0.5] (3,3) -- (4,3) -- (5,4) -- (5,5) -- (4,5) -- (3,4) -- cycle;
  \fill[green, opacity=0.5] (0,6) -- (2,6) -- (0,4) -- cycle;
  \fill[green, opacity=0.5] (0,1) -- (1,1) -- (2,2) -- (2,4) -- (1,4) -- (0,3) -- cycle;

  % Grid with diagonals
  \foreach \x in {0,...,5} {
      \foreach \y in {0,...,5} {
          \draw[thin] (\x,\y) rectangle (\x+1,\y+1);
          \draw[thin] (\x,\y) -- (\x+1,\y+1);
      }
  }

  % Circles at vertices
  \foreach \x in {0,...,6} {
      \foreach \y in {0,...,6} {
          \filldraw[black] (\x,\y) circle (2pt);
      }
  }

  % Crosses
  \foreach \x in {0,...,6} {
      \foreach \y in {0,...,6} {
          \draw[red, line width=0.5pt] (\x-0.05,\y-0.05) -- (\x+0.05,\y+0.05);
          \draw[red, line width=0.5pt] (\x-0.05,\y+0.05) -- (\x+0.05,\y-0.05);
      }
  }

  % Labels
  \node at (3.7, 4.2) {$z_k$};
  \node at (3.7, 3.2) {$\mathcal{K}_k$};
  \node at (0.4, 2.1) {$z_j$};
  \node at (0.3, 1.7) {$\mathcal{K}_j$};
  \node at (0.2, 5.8) {$z_i$};
  \node at (0.3, 4.7) {$\mathcal{K}_i$};
\end{scope}

% =====================
% Second picture (middle)
% =====================
\begin{scope}[shift={(6.5,0)}]  % <-- move to the right for spacing
  % (repeat your second picture code here)
  \fill[green, opacity=0.5] (3,3) -- (4,3) -- (5,4) -- (5,5) -- (4,5) -- (3,4) -- cycle;
  \fill[green, opacity=0.5] (0,6) -- (2,6) -- (0,4) -- cycle;
  \fill[green, opacity=0.5] (0,1) -- (1,2) -- (1,3) -- (0,3) -- cycle;

  \foreach \x in {0,...,5} {
      \foreach \y in {0,...,5} {
          \draw[thin] (\x,\y) rectangle (\x+1,\y+1);
          \draw[thin] (\x,\y) -- (\x+1,\y+1);
      }
  }

  \foreach \x in {0,...,6} {
      \foreach \y in {0,...,6} {
          \filldraw[black] (\x,\y) circle (2pt);
      }
  }

  \foreach \x in {0,...,12} {
      \foreach \y in {0,...,12} {
          \draw[red, line width=0.5pt] (0.5*\x-0.05,0.5*\y-0.05) -- (0.5*\x+0.05,0.5*\y+0.05);
          \draw[red, line width=0.5pt] (0.5*\x-0.05,0.5*\y+0.05) -- (0.5*\x+0.05,0.5*\y-0.05);
      }
  }

  \node at (3.7, 4.2) {$z_k$};
  \node at (3.7, 3.2) {$\mathcal{K}_k$};
  \node at (0.4, 2.1) {$z_j$};
  \node at (0.3, 1.7) {$\mathcal{K}_j$};
  \node at (0.2, 5.8) {$z_i$};
  \node at (0.3, 4.7) {$\mathcal{K}_i$};
\end{scope}

% =====================
% Legend (narrow, right side)
% =====================
\begin{scope}[shift={(13,1)}]  % <-- move further right
  \draw[black] (0,0) rectangle (3,2); % legend frame

  % Circle
  \filldraw[black] (0.3,1.6) circle (2pt);
  \node[anchor=west] at (0.6,1.6) {Vertex};

  % Cross
  \draw[red, line width=0.5pt] (0.2,1.0) -- (0.4,1.2);
  \draw[red, line width=0.5pt] (0.2,1.2) -- (0.4,1.0);
  \node[anchor=west] at (0.6,1.1) {Node};

  % Patch
  \fill[green, opacity=0.5] (0.2,0.2) rectangle (0.5,0.5);
  \node[anchor=west] at (0.6,0.35) {Patch};
\end{scope}

\end{tikzpicture}
    \caption{Triangular mesh. The local patches $\mathcal{K}_i,\; \mathcal{K}_j,$ and $\mathcal{K}_k$ for respectively the boundary vertices $z_i$ and $z_j$ and the internal vertex $z_k$ are shown in green. On the left we have the the case $r = 1$ and $d = 2$ and on the right $r = 2$ and $d = 2$}
    \label{fig:triangle_patch}
\end{figure}
\begin{figure}
    % \begin{tikzpicture}[scale=0.8]

% % Draw patch 1 (bottom-left corner)
% \draw[green] (0,0) rectangle (2,2);
% \draw[fill=green, fill opacity=0.5] (0,0) rectangle (2,2);

% % Draw patch 2 (center-left)
% \draw[fill=green, fill opacity=0.5] (0,3) rectangle (2,5);

% % Draw patch 3 (center-right)
% \draw[fill=green, fill opacity=0.5] (3,2) rectangle (5,4);

% % Draw the 6x6 grid
% \foreach \x in {0,...,6} {
%     \draw[thin] (\x,0) -- (\x,6);
% }
% \foreach \y in {0,...,6} {
%     \draw[thin] (0,\y) -- (6,\y);
% }

% % Draw a circle at each grid vertex
% \foreach \x in {0,...,6} {
%     \foreach \y in {0,...,6} {
%         \filldraw[black] (\x,\y) circle (2pt);
%     }
% }

% \foreach \x in {0,...,12} {
%     \foreach \y in {0,...,12} {
%         \draw[red, line width=0.5pt] (0.5*\x-0.05,0.5*\y-0.05) -- (0.5*\x+0.05,0.5*\y+0.05);
%         \draw[red, line width=0.5pt] (0.5*\x-0.05,0.5*\y+0.05) -- (0.5*\x+0.05,0.5*\y-0.05);
%     }
% }

% %Write text patch 1
% \node at (0.3, 0.2) {$z_i$};
% \node at (0.7, 1.3) {$\mathcal{K}_i$};

% %Write text patch 2
% \node at (0.3, 3.2) {$z_j$};
% \node at (0.7, 4.3) {$\mathcal{K}_j$};

% %Write text patch 2
% \node at (4.3, 3.2) {$z_k$};
% \node at (3.7, 2.3) {$\mathcal{K}_k$};

% \end{tikzpicture}

\begin{tikzpicture}[scale=0.8]

% =====================
% First picture (left)
% =====================
\begin{scope}[shift={(0,0)}]  % <-- left position
% Draw patch 1 (bottom-left corner)
    \draw[green] (0,0) rectangle (2,2);
    \draw[fill=green, fill opacity=0.5] (0,0) rectangle (2,2);
    
    % Draw patch 2 (center-left)
    \draw[fill=green, fill opacity=0.5] (0,3) rectangle (2,5);
    
    % Draw patch 3 (center-right)
    \draw[fill=green, fill opacity=0.5] (3,2) rectangle (5,4);
    
    % Draw the 6x6 grid
    \foreach \x in {0,...,6} {
        \draw[thin] (\x,0) -- (\x,6);
    }
    \foreach \y in {0,...,6} {
        \draw[thin] (0,\y) -- (6,\y);
    }
    
    % Draw a circle at each grid vertex
    \foreach \x in {0,...,6} {
        \foreach \y in {0,...,6} {
            \filldraw[black] (\x,\y) circle (2pt);
        }
    }
    
    %Write text patch 1
    \node at (0.3, 0.2) {$z_i$};
    \node at (0.7, 1.3) {$\mathcal{K}_i$};
    
    %Write text patch 2
    \node at (0.2, 4.2) {$z_j$};
    \node at (0.7, 4.3) {$\mathcal{K}_j$};
    
    %Write text patch 2
    \node at (4.3, 3.2) {$z_k$};
    \node at (3.7, 2.3) {$\mathcal{K}_k$};
    
    \foreach \x in {0,...,6} {
        \foreach \y in {0,...,6} {
            \draw[red, line width=0.5pt] (\x-0.05,\y-0.05) -- (\x+0.05,\y+0.05);
            \draw[red, line width=0.5pt] (\x-0.05,\y+0.05) -- (\x+0.05,\y-0.05);
        }
    }
\end{scope}

% =====================
% Second picture (middle)
% =====================
\begin{scope}[shift={(6.5,0)}]  % <-- move to the right for spacing
    % Draw patch 1 (bottom-left corner)
    \draw[green] (0,0) rectangle (2,2);
    \draw[fill=green, fill opacity=0.5] (0,0) rectangle (2,2);
    
    % Draw patch 2 (center-left)
    \draw[fill=green, fill opacity=0.5] (0,3) rectangle (2,5);
    
    % Draw patch 3 (center-right)
    \draw[fill=green, fill opacity=0.5] (3,2) rectangle (5,4);
    
    % Draw the 6x6 grid
    \foreach \x in {0,...,6} {
        \draw[thin] (\x,0) -- (\x,6);
    }
    \foreach \y in {0,...,6} {
        \draw[thin] (0,\y) -- (6,\y);
    }
    
    % Draw a circle at each grid vertex
    \foreach \x in {0,...,6} {
        \foreach \y in {0,...,6} {
            \filldraw[black] (\x,\y) circle (2pt);
        }
    }
    
    \foreach \x in {0,...,12} {
        \foreach \y in {0,...,12} {
            \draw[red, line width=0.5pt] (0.5*\x-0.05,0.5*\y-0.05) -- (0.5*\x+0.05,0.5*\y+0.05);
            \draw[red, line width=0.5pt] (0.5*\x-0.05,0.5*\y+0.05) -- (0.5*\x+0.05,0.5*\y-0.05);
        }
    }
    
    %Write text patch 1
    \node at (0.3, 0.2) {$z_i$};
    \node at (0.7, 1.3) {$\mathcal{K}_i$};
    
    %Write text patch 2
    \node at (0.2, 4.2) {$z_j$};
    \node at (0.7, 4.3) {$\mathcal{K}_j$};
    
    %Write text patch 2
    \node at (4.3, 3.2) {$z_k$};
    \node at (3.7, 2.3) {$\mathcal{K}_k$};
\end{scope}

% =====================
% Legend (narrow, right side)
% =====================
\begin{scope}[shift={(13,1)}]  % <-- move further right
  \draw[black] (0,0) rectangle (3,2); % legend frame

  % Circle
  \filldraw[black] (0.3,1.6) circle (2pt);
  \node[anchor=west] at (0.6,1.6) {Vertex};

  % Cross
  \draw[red, line width=0.5pt] (0.2,1.0) -- (0.4,1.2);
  \draw[red, line width=0.5pt] (0.2,1.2) -- (0.4,1.0);
  \node[anchor=west] at (0.6,1.1) {Node};

  % Patch
  \fill[green, opacity=0.5] (0.2,0.2) rectangle (0.5,0.5);
  \node[anchor=west] at (0.6,0.35) {Patch};
\end{scope}

\end{tikzpicture}
    \caption{Quadrilateral mesh. The local patches $\mathcal{K}_i,\; \mathcal{K}_j,$ and $\mathcal{K}_k$ for respectively the boundary vertices $z_i$ and $z_j$ and the internal vertex $z_k$ are shown in green. On the left we have the the case $r = 1$ and $d = 2$ and on the right $r = 2$ and $d = 2$}
    \label{fig:quad_patch}
\end{figure}
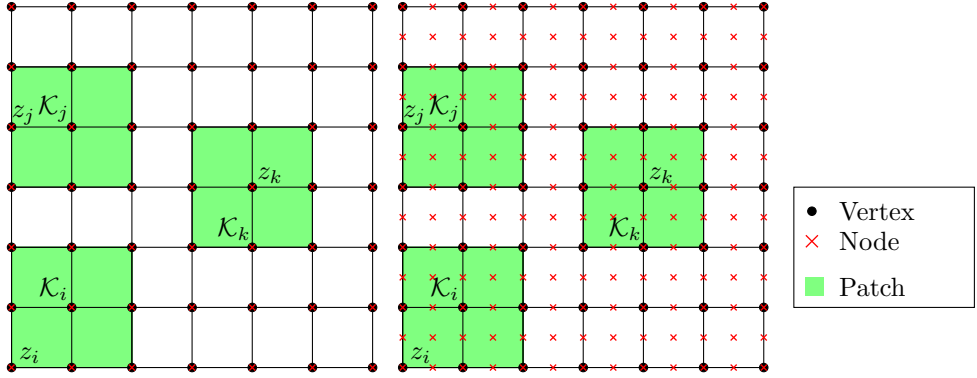
\noindent This concludes the first step. For the second step, we can characterise each node as a vertex node, an edge node, a face node, or an internal node. We first consider a vertex node $z_i$. The local polynomial of order $r+1$, denoted by $p_{z_i}$ is determined by a Discrete Least Squares Polynomial Approximation (DLSPA) over the nodal points that are inside the local patch $\mathcal{K}_i$, where the set of these nodal points is denoted by $\mathcal{N}_{z_i}$. 
\begin{equation}\label{eq:DLSPA}
\sum_{z\in\mathcal{N}_{z_i}}\left|\left(u_h-p_{z_i}\right)(z)\right|^2 \leq \sum_{z\in\mathcal{N}_{z_i}}\left|\left(u_h-p\right)(z)\right|^2, \quad \forall p\in \mathbb{P}_{r+1}(\mathcal{K}_{i}). 
\end{equation}
Again, for a quadrilateral and cuboid mesh, we replace $\mathbb{P}_{r+1}(\mathcal{K}_{i})$ by\\ $\mathbb{Q}_{r+1}(\mathcal{K}_{i})$. Then the recovered gradient $\mathrm{G}_h u_h$ at the vertex node $z_i$ is defined as
$$
(\mathrm{G}_h u_h)(z_i) = \left.\nabla p_{z_i}(x)\right|_{x = z_i}.
$$
If $z_i$ is an edge node, then we take the convex combination of the DLSPAs of the vertices associated with the edge, $z_{i_1}$ and $z_{i_2}$, and evaluate the gradient at the edge node:
$$
\left(\mathrm{G}_h u_h\right)\left(z_i\right)=\beta \left.\nabla p_{z_{i_1}}\left(x\right)\right|_{x = z_i}+(1-\beta) \left.\nabla p_{z_{i_2}}\left(x\right)\right|_{x = z_i},
$$
where $\beta$ is the ratio of the distances $z_i$ to $z_{i_1}$ and $z_{i_2}$. If we have a face node, we proceed analogously, but now the convex combination is taken of the DLSPAs of the vertices of the face. Lastly, if we have an internal node, we take the convex combinations of the DLSPAs of the vertices of the element in which the internal node is contained. This completes the definition of the recovered gradient. We denote the local error indicators $\{\eta_{K}: K\in \mathcal{T}_h\}$ and the global indicators $\eta_h$ defined as:
$$
\eta_{K} := \|\mathrm{G}_{\mathrm{h}} u_h-\nabla u_h\|_{L^2(K)} \text { and } \eta_h := \left(\sum_{K \in \mathcal{T}_h} \eta_{K}^2\right)^{\frac{1}{2}}
$$
It turns out that $\eta_h$ is a good approximation of $\|\nabla (u - u_h)\|_{L^2(\Omega)}$ if $\mathrm{G}_{\mathrm{h}}u_h$ is a better (``super converged") approximation of $\nabla u$ in $\Omega$.
\subsection{Properties}
The recovered gradient $\mathrm{G}_h$ has the two following properties: 
\begin{itemize}
    \item Consistency: 
    The PPR operator $\mathrm{G}_h$ preserves polynomials of degree $r + 1$ \cite[Theorem~2.1]{Zhang2005}, i.e.,
    \begin{equation}\label{eq:ppr}
    \mathrm{G}_h I_h^r p = \nabla p, \quad \forall p \in \mathbb{P}_{r+1}(\Omega).   
    \end{equation}
    Therefore, by the Bramble–Hilbert lemma \cite{bramble1971bounds}, there exists $C > 0$ such that, for all $h > 0$ and $u \in W_{\infty}^{r+2}(\Omega)$,
    \begin{equation}\label{eq:consistency}
    \| \nabla u - \mathrm{G}_h I_h^r u \|_{L^{2}(\Omega)} \le Ch^{r+1}.
    \end{equation}
    \item Boundedness: for all $K \in \mathcal{T}_h$,
    \begin{equation}\label{eq:bdd}
    \left\|\mathrm{G}_h v_h\right\|_{L^2(K)} \leq C \left\|\nabla v_h\right\|_{L^2(\Omega_K)},    
    \end{equation}
    where $\Omega_{K}$ is the union of patches $\mathcal{K}_i$ associated with the nodes $z_i$ in $K$ (see \cite[Section~3]{Naga2004}). 
\end{itemize}
From property \eqref{eq:consistency}, \eqref{eq:ppr} and the supercloseness property, namely that there exists a constant $C > 0$ and $\rho \in (0,1]$ such that, for all $h > 0$,
\begin{equation}\label{eq:supercloseness}
    \|\nabla(I_{h}^{r}u - u_h)\|_{L^2(\Omega)} \leq C h^{r+\rho}.
\end{equation} 
superconvergence follows: There exist a constant $C > 0$ and $\rho \in (0,1]$ such that, for all $h > 0$,
\begin{equation}\label{eq:superconvergence}
\|\nabla u - \mathrm{G}_{h} u_h\|_{L^{2}(\Omega)} \leq C h^{r+\rho}.    
\end{equation}
Indeed
\[
\|\nabla u - \mathrm{G}_{h} u_h\|_{L^{2}(\Omega)} \leq \|\nabla u - \mathrm{G}_{h}I_h^r u\|_{L^{2}(\Omega)} + \|\mathrm{G}_{h}I_h^r u - \mathrm{G}_{h} u_h\|_{L^{2}(\Omega)},
\]
where the first term in the right-hand side is bounded by \eqref{eq:consistency} and the second term is bounded because of \eqref{eq:bdd} and \eqref{eq:superconvergence}. The supercloseness property in \eqref{eq:supercloseness} is derived for a FEM solution $u_h$ of a second-order linear elliptic problem; see, for instance, \cite{wahlbin2006superconvergence}. Nevertheless, our numerical results (presented later) show that the recovered gradient in space and time is indeed superconvergent for the wave equation as well. The superconvergence of the recovered gradient $\mathrm{G}_h u_h$ provides the foundation for establishing the asymptotic exactness of the error indicator $\eta_h$ defined by $\eta_h = \|\nabla u_h - \mathrm{G}_hu_h\|_{L^2(\Omega)}$. Indeed, by the triangle inequality, we have the following two-sided estimate for the error:
\begin{multline*}
-\| \nabla u - \mathrm{G}_h u_h\|_{L^2(\Omega)} + \eta_h \leq \|\nabla(u - u_h)\|_{L^2(\Omega)}
\leq \| \nabla u - \mathrm{G}_h u_h\|_{L^2(\Omega)} + \eta_h.
\end{multline*}
It is clear now that if $\|\nabla u - \mathrm{G}_hu_h\|_{L^2(\Omega)}$ is superconvergent with respect to the true error $\|\nabla (u - u_h)\|_{L^2(\Omega)}$, this estimate suggest that $\eta_h$ and $\|\nabla (u - u_h)\|_{L^2(\Omega)}$ are of the same order. Moreover if we assume that the error converges with known order $r$, that is there exist constants $C_1, C_2 > 0$ such that, for all $h > 0$,  
\begin{equation}\label{eq:lower_bound}
 C_1 h^{r} \leq \|\nabla(u - u_h) \|_{L^2(\Omega)} \leq  C_2 h^{r},
\end{equation}
then using the superconvergence of the recovered gradient, \eqref{eq:superconvergence}, together with \eqref{eq:lower_bound}, we obtain  
\begin{equation}
    \left | \frac{\eta_h}{\|\nabla(u - u_h)\|_{L^2(\Omega)}} - 1 \right | \leq C_3 h^{\rho},
\end{equation}
where $C_3$ is independent of $h$.
%----------------------------------------------------------------------------------------------------------
%----------------------------------------------------------------------------------------------------------
%----------------------------------------------------------------------------------------------------------

\section{Applications for FDM}
\label{sec:main}
In this section, we introduce an error indicator for FDMs. The idea is to interpolate the Finite Difference (FD) solution to a suitable polynomial Finite Element space and thereafter indicate the error of the interpolated FD solution using the recovery-based error indicator introduced in Section \ref{sec:gradient_recovery}. We consider the poisson and the wave equation, but the method could be used for a general problem solved with an FDM.

\subsection{The Poisson Problem}\label{subsection:Poisson_Problem_theory}

We consider the 2D  poisson problem on the unit square $\Omega = (0,1)^2$, with $d = 2$:
\begin{equation}\label{eq:Poisson_equation}
 \begin{alignedat}{2}
-\Delta u(x_1,x_2)&= f(x_1,x_2),&\quad &(x_1,x_2)\in \Omega,\\
u(x_1,x_2) &= g(x_1,x_2),&\quad &(x_1,x_2)\in \partial \Omega,
\end{alignedat}
\end{equation}
where $f$ and $g$ are sufficiently smooth. Given a positive integer $N$, we define the grid spacing as $\frac{1}{N+1}$. We use a uniform Cartesian grid consisting of the FD grid points $\{(x_{1,i},x_{2,j}) = (\frac{i}{N+1},\frac{j}{N+1})\}_{i,j = 0}^{N+1}$. Let $u_{ij}$ be an approximation to $u(x_{1,i}, x_{2,j})$. We replace the Laplacian with a $r+1$-order centered FD formula to discretize equation \eqref{eq:Poisson_equation}. This gives rise to the following linear system:
\begin{equation}\label{eq:1D_Poisson:linsys_2}
A^{(r+1)} \mathbf{u} = \mathbf{f} + \mathbf{b},
\end{equation}
where $\mathbf{u} :=  (u_{11}; \hdots u_{N1}\; u_{12} \; u_{22}\;\hdots\; u_{NN})^T \in \mathbb{R}^{N^2}$, $\mathbf{f} :=  (f_{11}\; \hdots \newline f_{N1}\; f_{12} \; f_{22}\;\hdots\; f_{NN})^T \in \mathbb{R}^{N^2}$, and $A^{(r+1)} \in \mathbb{R}^{N^2\times N^2}$. The vector $\mathbf{b}\in \mathbb{R}^{N^2}$ handles the boundary conditions. If $u$ is sufficiently smooth (i.e., for $r=3$ we require $u \in C^6(\overline{\Omega})$), then there exists a $C>0$ independent of $N$ such that for all $N>0$ (see for instance \cite[Chapter~3]{Leveque2007}):
\begin{equation}\label{eq:1D_Poisson:FDM_estimate}
\max_{1 \leq i,j \leq N}|u_{ij} - u(x_{1,i},x_{2,j})| \leq C \left(\frac{1}{N+1}\right)^{r+1},
\end{equation}
Next, we discuss how to use the recovery-based error indicator of Section \ref{sec:gradient_recovery}. We interpolate the FD solution $u_{i,j}$ at the FD grid points to obtain a function $u_h(x_1,x_2)$ that approximates $u(x)$. We start by assuming that $N+1$ is a multiple of $r$ and define an ``interpolant mesh". We select a subset of FD grid points $\{(x_{1,ir}, x_{2,jr})\}_{i,j = 0}^{\frac{N+1}{r}}$. This subset of points serves as the mesh vertices of the interpolant mesh, which consists of rectangles with width and height $\displaystyle h:= \frac{r}{N+1}$ (that thus depends upon $r$) and forms a partition of $\overline{\Omega}$, 
$$\overline{\Omega} = \bigcup_{i,j = 0}^{\frac{N+1}{r}-1}Q_{ij},$$ 
where $Q_{ij} :=(x_{1, ri}, x_{1, r(i+1)}) \times (x_{2, rj}, x_{2, r(j+1)})$. Given the interpolant mesh, we define the continuous piecewise polynomial space of degree $r$:
\begin{equation*}
\mathcal{V}_{h}^r := \left\{v \in C^0(\Omega) : v|_{Q_{ij}} \in \mathbb{Q}_r(Q_{ij}), \quad i,j = 1, \hdots, \frac{N+1}{r}-1\right\}.    
\end{equation*}
A basis for $\mathcal{V}_h^r$ is the standard tensor-product Lagrange basis $\{\phi_i(x_1)\phi_j(x_2)\}_{i,j = 0}^{N+1}$, where $\phi_{i}(x_{1,l}) = \delta_{i l},\text{ where } 0 \leq i,l \leq N+1$ and $\phi_{j}(x_{2,l}) = \delta_{j l},\text{ where } 0 \leq j,l \leq N+1$. The interpolant mesh nodes coincide with the FD grid points,
\begin{figure}
    \begin{tikzpicture}[scale=1, thick]

% Grid setup
\def\nfine{6}     % Number of FD intervals
\def\ncoarse{2}   % Number of FE elements
\def\rh{3}        % Coarsening ratio

% FD Grid (fine)
\foreach \i in {0,...,\nfine} {
    \draw[gray, dashed] (\i,0) -- (\i,\nfine);
    \draw[gray, dashed] (0,\i) -- (\nfine,\i);
}

% FE Grid (coarse)
\foreach \i in {0,...,\ncoarse} {
    \pgfmathsetmacro{\x}{\i*\rh}
    \draw[blue, very thick] (\x,0) -- (\x,\nfine);
    \draw[blue, very thick] (0,\x) -- (\nfine,\x);
}

% Coarse Element Labels
\node[blue] at (1.5,1.5) {\large $Q_{0,0}$};
\node[blue] at (4.5,1.5) {\large $Q_{1,0}$};
\node[blue] at (1.5,4.5) {\large $Q_{0,1}$};
\node[blue] at (4.5,4.5) {\large $Q_{1,1}$};

% Axes
\draw[->, thick] (-0.25,0) -- (\nfine+0.5,0) node[right] {$x_1$};
\draw[->, thick] (0,-0.25) -- (0,\nfine+0.5) node[above] {$x_2$};

% x_1 ticks
\foreach \i in {0,...,\nfine} {
    \node[below=1pt] at (\i,-0.1) {\scriptsize $x_{1,\i}$};
}

% x_2 ticks
\foreach \j in {0,...,\nfine} {
    \node[left] at (-0.1,\j) {\scriptsize $x_{2,\j}$};
}

% Basis function labels (smaller, correctly oriented)
\foreach \i in {0,...,\nfine} {
    \node[below=1pt] at (\i,-1) {\tiny $\phi_{\i}$};
    \node[left=1pt] at (-3,\i) {\tiny $\phi_{\i}$};
}

% Define Basis functions (on reference interval [0,3])
\pgfmathdeclarefunction{phi0}{1}{\pgfmathparse{-9*#1^3/2 + 9*#1^2 - 11/2*#1 + 1}}
\pgfmathdeclarefunction{phi1}{1}{\pgfmathparse{#1*(9*#1^2 - 9*#1 + 2)/2}}
\pgfmathdeclarefunction{phi2}{1}{\pgfmathparse{9*#1*(3*#1^2 - 5*#1 + 2)/2}}
\pgfmathdeclarefunction{phi3}{1}{\pgfmathparse{9*#1*(-3*#1^2 + 4*#1 - 1)/2}}

% Horizontal basis functions (bottom row)
\begin{scope}[samples=100, variable=\x, blue]
    \foreach \k in {0,...,1} { % Only two coarse elements
        % phi0 is supported on [0,3], so plot it in local coordinates
        \draw[domain=0:1, shift={(\k*3,-3)}, thin]
            plot ({\x*3}, {phi0(\x)});
        \draw[domain=0:1, shift={(\k*3,-3)}, thin]
            plot ({\x*3}, {phi1(\x)});
        \draw[domain=0:1, shift={(\k*3,-3)}, thin]
            plot ({\x*3}, {phi2(\x)});
        \draw[domain=0:1, shift={(\k*3,-3)}, thin]
            plot ({\x*3}, {phi3(\x)});
    }
\end{scope}

% Vertical basis functions (bottom row)
\begin{scope}[samples=100, variable=\x, blue]
    \foreach \k in {0,...,1} { % Only two coarse elements
        % phi0 is supported on [0,3], so plot it in local coordinates
        \draw[domain=0:1, shift={(-1.25, \k*3)}, thin, rotate=90]
            plot ({\x*3}, {phi0(\x)});
        \draw[domain=0:1, shift={(-1.25, \k*3)}, thin, rotate=90]
            plot ({\x*3}, {phi1(\x)});
        \draw[domain=0:1, shift={(-1.25, \k*3)}, thin, rotate=90]
            plot ({\x*3}, {phi2(\x)});
        \draw[domain=0:1, shift={(-1.25, \k*3)}, thin, rotate=90]
            plot ({\x*3}, {phi3(\x)});
    }
\end{scope}

% h and rh spacing
\draw[|-|] (0.05,0.3) -- (0.95,0.3) node[midway, fill=white, inner sep=1pt] {\scalebox{0.5}{$\frac{1}{N+1}$}};
\draw[|-|] (0.1,0.9) -- (2.9,0.9) node[midway, fill=white, inner sep=1pt] {\scriptsize $h = \frac{r}{N+1}$};

% Axes horizontal
\draw[->, thick] (-0.25,-3) -- (\nfine+0.5,-3) node[right] {$x_1$};
\draw[->, thick] (0,-3.25) -- (0,-1.5);

% Axes vertical
\draw[->, thick] (-1,0) -- (-3,0);
\draw[->, thick] (-1.25,-0.25) -- (-1.25,\nfine+0.5) node[above] {$x_2$};

\end{tikzpicture}
    \caption{Interpolant mesh (blue solid lines) and FD grid (gray dotted lines) for $N = 5$ and $r = 3$ together with the one-dimensional cubic Lagrange basis functions in each spatial direction}
    \label{fig:method:poisson_problem:fd_grid_fe_mesh}
\end{figure}
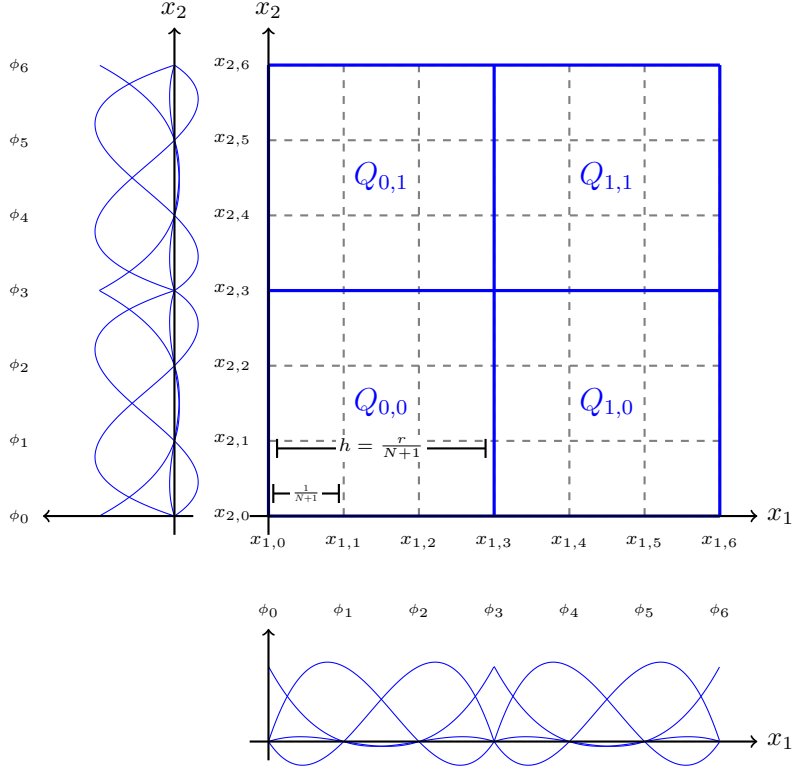
for example consider $N = 5$ and $r = 3$ in Figure \ref{fig:method:poisson_problem:fd_grid_fe_mesh}. The gray dotted line is the FD grid with grid size $\frac{1}{N+1}$ consisting of the FD grid points $\{(x_{1,i},x_{2,j})\}_{i,j = 0}^{6}$ and the blue solid line is the interpolant mesh consisting of the elements $\{Q_{00},\; Q_{10},\; Q_{01},\; Q_{11} \}$ which are rectangles with width and height $h$. Additionally, the cubic Lagrange basis functions are plotted for each spatial direction. We observe that associated with each tensor product cubic Lagrange basis function, we have an FD grid point $(x_{1,i}, x_{2,j})$ and therefore an FD value $u_{ij}$. In general, this allows us to write the interpolated FD solution as:
\begin{equation*}
u_h(x_1,x_2) = \sum_{i = 0}^{N+1} \sum_{j = 0}^{N+1} u_{ij}\phi_{i} (x_1)\phi_{j}(x_2),\quad (x_1,x_2)\in\overline{\Omega}.
\end{equation*}
The following proposition states the a-priori error estimates for the interpolated FD solution $u_h$.
\begin{proposition}\label{thm:convergence_bounds}
Let $u$ be the (sufficiently smooth) solution of problem \eqref{eq:Poisson_equation}.
Suppose that $u_{ij}$, for $i,j = 0, \hdots, N+1$ is a finite difference solution of order $(r+1)$, i.e.,
\[
\max_{1 \leq i,j \leq N} \big|u_{ij} - u(x_{1,i},x_{2,j})\big| 
   \leq C_1 \left(\frac{1}{N+1}\right)^{r+1}.
\]
Let $u_h$ denote the Lagrange interpolant of order $r$ defined on the interpolant mesh with meshsize 
$h = \tfrac{r}{N+1}$.
Then there exists a constant $C>0$, independent of $h$, such that
\begin{align}\label{eq:conv_bound}
    \left\|\nabla(u-u_h)\right\|_{L^2(\Omega)} \leq C h^{r}.
\end{align}
\end{proposition}
\begin{proof}
Let $I_{h}^{r}u$ be the $r$-th order Lagrange interpolant of $u$ on the interpolant mesh introduced above. We decompose the error as follows:  
\begin{equation*}
    \left\|\nabla(u-u_h)\right\|_{L^2(\Omega)} 
    \leq \left\|\nabla(u-I_{h}^{r} u)\right\|_{L^2(\Omega)}
    + \left\|\nabla(I_{h}^{r} u-u_h)\right\|_{L^2(\Omega)}.
\end{equation*}
The first term on the right-hand side is a classical interpolation estimate: there exists a constant $C_1>0$, independent of $h$, such that  
\[
\left\|\nabla(u-I_{h}^{r} u)\right\|_{L^2(\Omega)}
    \leq C_1 h^{r}\,|u|_{H^{r+1}(\Omega)},
\]
see, for example, \cite[Chap.~3.1]{ciarlet2002finite}. For the second term, we first apply the inverse inequality: there exists a constant $C_1>0$, independent of $h$, such that  
\[
\left\|\nabla(I_{h}^{r} u-u_h)\right\|_{L^2(\Omega)} \leq C_2 h^{-1}\left\|I_{h}^{r} u-u_h\right\|_{L^2(\Omega)}.
\]
Since  
\[
\left\|I_{h}^{r} u-u_h\right\|_{L^2(\Omega)} \leq C \max_{1 \leq i,j \leq N}\big|u_{ij} - u(x_{1,i},x_{2,j})\big|\left\|\sum_{i = 0}^{N+1} \sum_{j = 0}^{N+1} \phi_{i} (x_1)\phi_{j}(x_2)\right\|_{L^2(\Omega)},
\]
\noindent then, because of \eqref{eq:1D_Poisson:FDM_estimate}:
\[
\left\|\nabla(I_{h}^{r} u-u_h)\right\|_{L^2(\Omega)} \leq C_2 h^{-1}C\left(\frac{1}{N+1}\right)^{r+1}\sqrt{\Omega} \leq C_3 h^r, 
\]
where $C_3>0$ is a constants independent of $h$. Combining these two estimates yields the desired bound \eqref{eq:conv_bound}, which completes the proof.
\end{proof}
\noindent We estimate the error $\|\nabla (u - u_h)\|_{L^2(\Omega)}$ using the recovery based error indicator introduced in Section \ref{sec:gradient_recovery}, namely:
\begin{equation}\label{eq:indicator_poisson_global}
\eta_{h} := \left(\sum_{i = 0}^{\frac{N+1}{r}-1}\sum_{j = 0}^{\frac{N+1}{r}-1} \eta_{Q_{ij}}^2\right)^{\frac{1}{2}},
\end{equation}
where the local error indicators, $\{ \eta_{Q_{ij}} \}_{i,j = 0}^{\frac{N+1}{r}-1}$, are defined as
\begin{equation}\label{eq:indicator_poisson_local}
\eta_{Q_{ij}} := \|\mathrm{G}_{h}u_h - \nabla u_h\|_{L^2(Q_{ij})},
\end{equation}
where $\mathrm{G}_{h}: \mathcal{V}_{h}^{r} \rightarrow \mathcal{V}_{h}^{r}\times \mathcal{V}_{h}^{r}$ is the PPR operator on the interpolant mesh and $\mathrm{G}_{h}u_h$ is the recovered gradient with components $\mathrm{G}_{h} u_h := \begin{pmatrix}\mathrm{G}_{h}^{x_1} u_h &\mathrm{G}_{h}^{x_2} u_h\end{pmatrix}^T$. At this point we would like to stress that in Section \ref{sec:gradient_recovery}, $u_h$ is the solution of a linear elliptic problem solved with the FEM, whereas in this section $u_h$ is not. Nevertheless, the numerical results in Subsection \ref{subsection:Poisson_Problem} indicate that the estimator still performs effectively.

\subsection{The Wave Equation}
We define $\Omega=(0,1)^2\times(0,T)$, with $d = 3$. We consider the 2D wave equation with constant wave speed on the unit square: 
\begin{equation}\label{eq:Wave_1D:problem}
\begin{alignedat}{2}
&\frac{\partial^2 u}{\partial t^2}(x_1,x_2,t) - \Delta u(x_1,x_2,t) = 0, 
&\quad &(x_1,x_2,t) \in \Omega, \\
&u(x_1,x_2,t) = 0, 
&\quad &(x_1,x_2,t) \in \partial\big((0,1)^2\big)\times [0,T], \\
&u(x_1,x_2,0) = f(x_1,x_2), 
&\quad &(x_1,x_2) \in (0,1)^2, \\
&\frac{\partial u}{\partial t}(x_1, x_2, 0) = g(x_1,x_2), 
&\quad &(x_1,x_2) \in (0,1)^2.
\end{alignedat}
\end{equation}
where $f,g$ are sufficiently smooth. The method described in this subsection can be easily extended to higher spatial dimensions and varying coefficients. Given positive integers $N_x$ and $N_t$, we define the grid spacing $\frac{1}{N_x+1}$ and time stepsize $\frac{T}{N_t+1}$. We have a uniform cartesian grid in space and time consisting of FD grid points $\{(x_{1,i}, x_{2,j}, t^n)\}_{i,j = 0}^{N_x+1}{}_{n = 0}^{N_t+1}$. We solve problem \eqref{eq:Wave_1D:problem} with a $r+1$ order FDM to obtain a FD solution $u_{ij}^n$ approximating $u(x_{1,i}, x_{2,j}, t^n)$, where $0 \leq i,j \leq N_x+1$ and $0 \leq n \leq N_t + 1$. If $u$ is sufficiently smooth then there exists an $C_1,\; C_2>0$ independent of $N_x$ and $N_t$ such that: 
\begin{equation}\label{eq:Wave:FDM_estimate}
    \max_{0 \leq i,j \leq N_x+1}|u_{ij}^{n} - u(x_{1,i},x_{2,j},t^{n}) | \leq C_1\left(\frac{1}{N_x+1}\right)^{r+1} + C_2\left(\frac{T}{N_t+1}\right)^{r+1},
\end{equation}
where $0 \leq n \leq N_t+1$. With the same line of reasoning as for the poisson problem, we assume that $N_x+1$ and $N_t+1$ are multiples of $r$ and construct an ``interpolant mesh" in space and time. We select a subset of FD grid points $\{(x_{1,ir}, x_{2,jr}, t^{nr})\}_{i,j = 0}^{\frac{N_x+1}{r}}{}_{n = 0}^{\frac{N_t+1}{r}}$. This subset of points serve as the mesh vertices of the interpolant mesh which consists of cuboids with sides $\frac{r}{N_x+1}$, $\frac{r}{N_x+1}$ and $\frac{rT}{N_t+1}$ and forms a partition of $\Omega$, i.e. 
$$\bigcup_{i,j=0}^{\frac{N_x+1}{r}-1}\!\bigcup_{n=0}^{\frac{N_t+1}{r}-1} Q_{ij}^n$$
where 
$$Q_{ij}^n:=(x_{1, ri}, x_{1, r(i+1)}) \times (x_{2, rj}, x_{2, r(j+1)})\times(t^{rn},t^{r(n+1)}).$$
We define the mesh parameter of the interpolant mesh to be 
$$h:= \max{\left\{\frac{r}{N_x+1}, \frac{rT}{N_t+1}\right\}}.$$ 
Given the interpolant mesh, we define the continuous piecewise polynomial space of degree $r$:
\begin{multline*}
    \mathcal{V}_{h}^r := \left\{v \in C^0(\Omega) : v|_{Q_{ij}^n} \in \mathbb{Q}_r(Q_{ij}^n), \quad i,j = 0,\hdots, \frac{N_x+1}{r}-1,  \right. \\ \left. n = 0,\hdots,\frac{N_t+1}{r}-1\right\}.
\end{multline*}
As mentioned in Subsection \ref{subsection:Poisson_Problem_theory}, a basis for the space $\mathcal{V}_{h}^r$ is the standard tensor-product Lagrange basis $\{\phi_i(x_1)\phi_j(x_2)\phi^n(t)\}_{i,j = 0}^{N_x+1}{}_{n = 0}^{N_t+1}$. The mesh nodes associated with the interpolant mesh coincide with the FD grid points, allowing us to define:
\begin{equation*}
u_{h}(x_1,x_2,t) = \sum_{i = 0}^{N_x+1} \sum_{j = 0}^{N_x+1} \sum_{n = 0}^{N_t+1} u_{ij}^n\phi_{i} (x_1)\phi_{j}(x_2)\phi^{n} (t),\quad (x_1,x_2,t)\in\Omega.
\end{equation*}
\noindent The following proposition states the a-priori error estimates for the interpolated FD solution $u_h$, analogously to Proposition \ref{thm:convergence_bounds} for the poisson problem.
\begin{proposition}\label{thm:convergence_bounds_wave}
Let $u$ be the solution of problem \eqref{eq:Wave_1D:problem}, with 
$u$ sufficiently smooth. 
Suppose that $u_{ij}^n$, for $i,j = 0, \hdots, N_x+1$ and $n = 0, \hdots, N_t+1$ is a finite difference solution of order $(r+1)$, i.e.,
\[
\max_{0 \leq i,j \leq N_x+1} 
\big|u_{ij}^{n} - u(x_{1,i},x_{2,j},t^{n}) \big|
   \leq C_1 \left(\tfrac{1}{N_x+1}\right)^{r+1}
     + C_2 \left(\tfrac{T}{N_t+1}\right)^{r+1}.
\]
Let $u_h$ denote the $r$-th order Lagrange interpolant of $u_{ij}^n$, where $i,j = 0, \hdots, N_x+1$ and $n = 0, \hdots, N_t+1$
defined on a interpolant space time mesh with meshsize $h = \max\left\{\tfrac{r}{N_x+1}, \tfrac{rT}{N_t+1}\right\}$. 
Then there exists a constant $C>0$, independent of $h$, such that
\begin{align}\label{eq:a_priori_wave}
    |||e_h|||
    := \left(
        \|\nabla e_h\|_{L^2(\Omega)}^2 
        + \Big\|\tfrac{\partial e_h}{\partial t}\Big\|_{L^2(\Omega)}^2
    \right)^{\tfrac{1}{2}}
    \;\leq\; C h^{r},
\end{align}
where $e_h := u - u_h$.
\end{proposition}
\begin{proof}
    The proof is analogous to the proof of Proposition \ref{thm:convergence_bounds}.
\end{proof}
\noindent The a-priori error estimate is stated in a norm that is not natural for the wave equation. 
In the literature \cite{picasso2010numerical,gorynina2019easily}, similar non-natural norms are also considered: 
for example, only the term $\|\nabla e_h(T)\|_{L^2((0,1)^2)}$ is used, since this is the natural norm for parabolic problems 
and corresponding error estimators are available. 
In analogy with the  poisson problem, we introduce the recovery-based error indicator from Section \ref{sec:gradient_recovery}, 
which allows us to measure the error $|||e_h|||^2$. We split the indicator into spatial and temporal contributions:
\begin{multline}
    \eta_{h}^2 = \left({\eta_{h}^x}\right)^2 + \left({\eta_{h}^t}\right)^2\\
    := \sum_{i = 0}^{\frac{N_x+1}{r}-1}\sum_{j = 0}^{\frac{N_x+1}{r}-1}\sum_{n = 0}^{\frac{N_t+1}{r}-1}{\eta_{Q_{ij}^n}^x}^2 
    + \sum_{i = 0}^{\frac{N_x+1}{r}-1}\sum_{j = 0}^{\frac{N_x+1}{r}-1}\sum_{n = 0}^{\frac{N_t+1}{r}-1}{\eta_{Q_{ij}^n}^t}^2,
\end{multline}
where $\{ \eta_{Q_{ij}^n}^x \}_{i,j = 0}^{\frac{N_x+1}{r}-1}{}_{n = 0}^{\frac{N_t+1}{r} - 1}$ are the local error indicators 
for the spatial gradient of the error, and 
$\{ \eta_{Q_{ij}^n}^t \}_{i,j = 0}^{\frac{N_x+1}{r}-1}{}_{n = 0}^{\frac{N_t+1}{r} - 1}$ are the local error indicators 
for the temporal derivative of the error, defined as
\begin{align}\label{eq:indicator_wave_local}
    \eta_{Q_{ij}^x} := \|\mathrm{G}_{h}^x u_h - \nabla u_h\|_{L^2(Q_{ij}^n)}, 
    \quad
    \eta_{Q_{ij}^t} := \|\mathrm{G}_{h}^t u_h - \partial_t u_h\|_{L^2(Q_{ij}^n)}.
\end{align}
\noindent Since $\Omega = (0,1)^2\times (0,T)$ the norm in \eqref{eq:a_priori_wave} is global in time; therefore, the indicator \eqref{eq:indicator_wave_local} will not provide any local information about the temporal error distribution.
The natural norm of the wave equation, on the other hand, is defined at a given time, typically at the final time $T$ (or a given time $T$):
$$
    |||e_{h}(x_1,x_2,T)|||^2 := \left\| \frac{\partial e_{h}}{\partial t}(x_1,x_2,T) \right\|^2_{L^2((0,1)^2)} 
    + \left\|\nabla e_{h}(x_1,x_2,T) \right\|^2_{L^2((0,1)^2)}.
$$
However, the recovery-based indicator introduced in Section \ref{sec:gradient_recovery} is defined locally in space-time that is, on an element or time slab and not for a single time instant $t \in [0,T]$.
To relate the indicator to the natural energy norm, we therefore regularize the latter by taking the Root Mean Square (RMS) over the final time slab.
We denote the RMS of a function $f$ by an overline for ease of notation:
$$
    \overline{f} = \left(\frac{1}{r\Delta t}\int_{T-r\Delta t}^T f(\tau)^2 \, d\tau\right)^{\frac{1}{2}},
$$
where, $\Delta t = \frac{T}{N_t+1}$. So the RMS natural norm is given by
\begin{equation}\label{eq:natural_norm_wave}
    \overline{|||e_{h}|||}^2 :=  \overline{\left\| \frac{\partial e_{h}}{\partial t} \right\|}^2_{L^2((0,1)^2)} 
    + \overline{\left\|\nabla e_{h} \right\|}^2_{L^2((0,1)^2)}.
\end{equation}
Moreover, we introduce an auxiliary partition of the spatial domain with elements, $Q_{ij}$:
$$
    (0,1)^2 = \bigcup_{i,j = 0}^{\frac{N_x+1}{r}-1} Q_{ij}, 
    \qquad Q_{ij}:=(x_{1, ri}, x_{1, r(i+1)}) \times (x_{2, rj}, x_{2, r(j+1)}).
$$
The corresponding RMS error indicator is
\begin{equation}\label{eq:wave_indicator_global}
    \overline{\eta_{h}}^2 := \overline{\eta_{h}^t}^2 + \overline{\eta_{h}^x}^2 
    := \sum_{i=0}^{\frac{N_x+1}{r}-1} \sum_{j = 0}^{\frac{N_x+1}{r}-1} \overline{\eta^t_{Q_{ij}}}^2 
    + \sum_{i=0}^{\frac{N_x+1}{r}-1} \sum_{j = 0}^{\frac{N_x+1}{r}-1} \overline{\eta^x_{Q_{ij}}}^2.
\end{equation}
Here, $\{ \overline{\eta_{Q_{ij}}^x} \}_{i,j = 0}^{\tfrac{N_x+1}{r}}$ are the local error indicators for the RMS of the spatial gradient of the error, and $\{\overline{\eta_{Q_{ij}}^t}\}_{i,j = 0}^{\tfrac{N_x+1}{r}}$ are the local error indicators for the RMS of the temporal derivative of the error, defined as
\begin{equation}\label{eq:local_error_indicator_wave}
\begin{aligned}
    \overline{\eta^x_{Q_{ij}}}^2 
    &:= \frac{1}{r\Delta t} \int_{T- r\Delta t}^T \|\mathrm{G}^x_{h}u_{h}(x_1,x_2,\tau) - \nabla u_{h}(x_1,x_2,\tau)\|_{L^2(Q_{ij})}^2 \, d\tau,\\
    \overline{\eta^t_{Q_{ij}}}^2 
    &:= \frac{1}{r\Delta t} \int_{T- r\Delta t}^T \left\|\mathrm{G}^t_{h}u_{h}(x_1,x_2,\tau) - \partial_\tau u_{h}(x_1,x_2,\tau)\right\|_{L^2(Q_{ij})}^2 \, d\tau.
\end{aligned}
\end{equation}
We will demonstrate the accuracy of this indicator in the next section.
%----------------------------------------------------------------------------------------------------------
%----------------------------------------------------------------------------------------------------------
%----------------------------------------------------------------------------------------------------------
\section{Numerical Experiments}
\label{sec:experiments}
In this section, we show the results for several numerical experiments. We first consider the poisson problem and then the wave equation. We aim to evaluate the performance of the error indicator in several test cases. We consider a uniform refinement in space and time (if applicable). Moreover, we introduce the effectivity index (e.i.), which is the ratio between the indicator and the error: 

\[
\text{e.i.} = \frac{\text{indicator}}{\text{error}}.
\]
\subsection{The Poisson Problem}\label{subsection:Poisson_Problem}

This section considers the  poisson problem on the unit square ($\Omega = (0,1)^2,\;d = 2$); see equation \eqref{eq:Poisson_equation}. We consider three numerical examples: an oscillatory solution, a solution with steep gradients, and a solution with a discontinuous gradient. We solve the first two examples with a second—and fourth-order FDM, i.e., the cases $r = 1, 3$, and the last example for $r = 1$. For $r = 1$, we consider a classical standard second-order FD scheme, and for the case $r = 3$, we consider the following difference formula to discretize the Laplacian:
\begin{multline*}
     - \Delta u (x_{1,i}, x_{2,j}) \approx \frac{u_{(i-2)j}-16u_{(i-1)j} + 30u_{ij} - 16u_{(i+1)j} + u_{(i+2)j}}{\left(\frac{1}{N+1}\right)^2}\\ + \frac{u_{i(j-2)}-16u_{i(j-1)} + 30u_{ij} - 16u_{i(j+1)} + u_{i(j+2)}}{\left(\frac{1}{N+1}\right)^2},\quad 1 < i,j < N.        
\end{multline*}
At the boundary, we use one-sided FD formulas. We approximate the error \\ $\|\nabla (u - u_h)\|_{L^2(\Omega)}$ using the indicator, $\eta_h$, defined in equation \eqref{eq:indicator_poisson_global} globally and equation \eqref{eq:indicator_poisson_local} locally. For the first two numerical examples, we know the analytical solution, and in the last numerical experiment, we have a reference solution on a fine mesh. 

\paragraph{Numerical example 1}

The exact solution is given by
$$
    u(x_1,x_2) = \sin{8\pi x_1}\sin{8\pi x_2}.
$$
The solution is smooth and exhibits oscillations in both spatial directions. In Figure \ref{fig:Poisson:oscillations}, we plot the error in the natural norm, alongside the indicator and the error of the recovered gradient for $r = 1$ and $3$. We observe that for an order \(r+1\) FD solution, the gradient of the error converges with order \(r\), which is consistent with the expectations from Proposition \ref{thm:convergence_bounds}. Moreover, the effectivity index converges to $1$, so the indicator is asymptotically exact. This is because the recovered gradient is superconvergent with an order of \(r+1\). In Figures \ref{fig:Poisson_const_coeff:oscillations:loc_err_k_2} and \ref{fig:Poisson_const_coeff:oscillations:loc_err_k_4}, the gradient of the error and the indicator are shown locally for \(r = 1\) and \(r = 3\), respectively. We observe that, locally, the features of the error are captured accurately by the error indicator. However, for \(r = 3\), we notice that the boundary features of the error are captured less accurately by the indicator.

\begin{figure}
  \begin{subfigure}[b]{0.49\textwidth}
    \includegraphics[width=\linewidth]{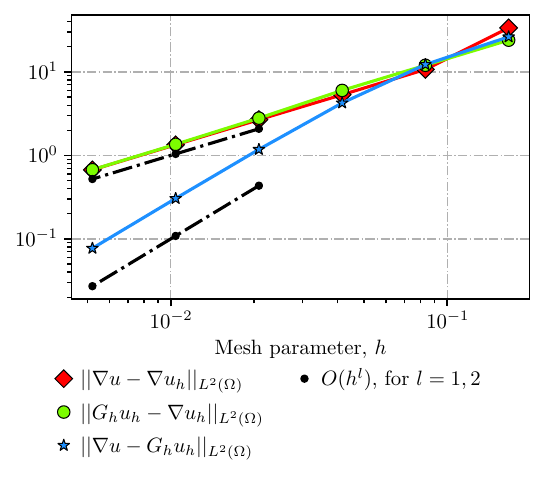}
  \end{subfigure}
  \hfill
  \begin{subfigure}[b]{0.49\textwidth}
    \includegraphics[width=\textwidth]{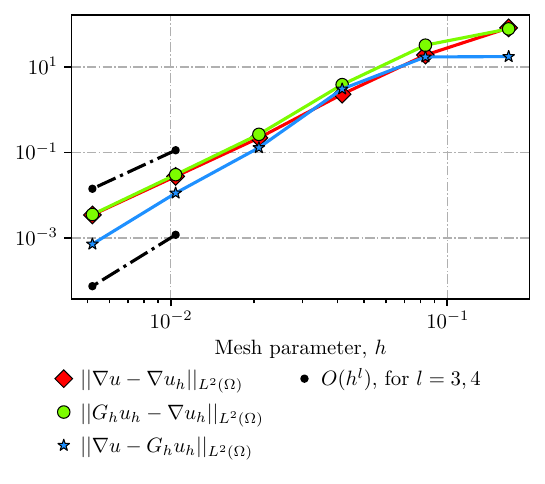}
  \end{subfigure}
\caption{The gradient of the error, the indicator, the error of the recovered gradient, for the poisson problem with exact solution $ u(x_1,x_2) = \sin{8\pi x_1}\sin{8\pi x_2}$ and various values of $h$, defined through $N+1$. For $r = 1$ on the left, $r = 3$ on the right and $N+1 = 6, 12, 24, 48, 96$ and $192$.}
\label{fig:Poisson:oscillations}
\end{figure}

\begin{figure}
    \includegraphics[width = \linewidth]{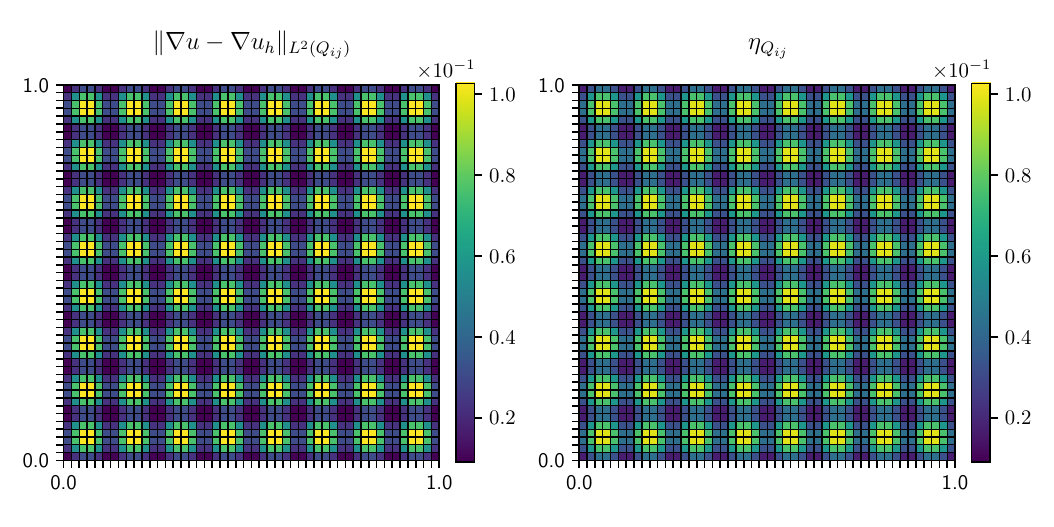}
    \caption{Local results for $r = 1$ and $h = \frac{1}{48}$ for the poisson problem with exact solution $ u(x_1,x_2) = \sin{8\pi x_1}\sin{8\pi x_2}$. On the left is the local error, and on the right is the local indicator. Additionally, we used the same color scale in both plots.}
  \label{fig:Poisson_const_coeff:oscillations:loc_err_k_2}
\end{figure}

\begin{figure}
    \includegraphics[width = \linewidth]{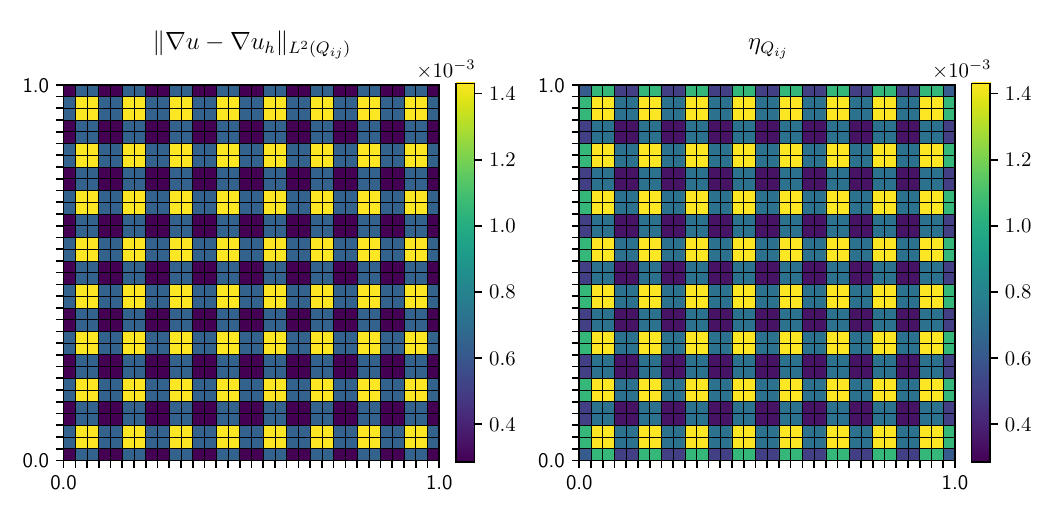}
    \caption{Local results for $r = 3$ and $h = \frac{1}{32}$ for the poisson problem with exact solution $ u(x_1,x_2) = \sin{8\pi x_1}\sin{8\pi x_2}$. On the left is the local error, and on the right is the local indicator. Additionally, we used the same color scale in both plots.}
    \label{fig:Poisson_const_coeff:oscillations:loc_err_k_4}
\end{figure}

\paragraph{Numerical example 2}

In this example, the exact solution is given by
$$
u(x_1,x_2) = \arctan{\left(25(\sqrt{x_1^2+x_2^2} - 0.5)\right)}.
$$
The solution has a steep gradient. In Figure \ref{fig:Poisson:steep_gradient}, we plot the error in the natural norm, alongside the indicator and the error of the recovered gradient for $r = 1$ and $3$. In Figures \ref{fig:Poisson_const_coeff:steep_gradient:loc_err_k_2} and \ref{fig:Poisson_const_coeff:steep_gradient:loc_err_k_4}, the gradient of the error and the indicator are shown locally for \(r = 1\) and \(r = 3\), respectively. We draw the same conclusions as in the previous example.

\begin{figure}
  \begin{subfigure}[b]{0.49\textwidth}
    \includegraphics[width=\linewidth]{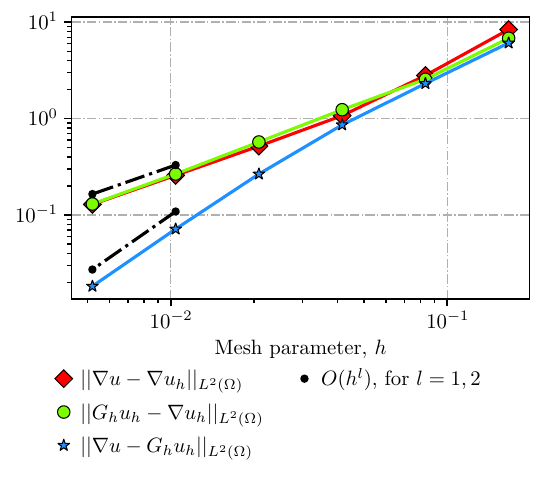}
  \end{subfigure}
  \hfill
  \begin{subfigure}[b]{0.49\textwidth}
    \includegraphics[width=\textwidth]{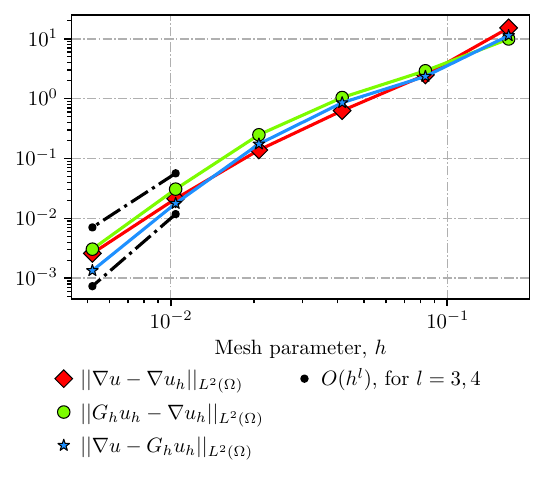}
  \end{subfigure}
\caption{The gradient of the error, the indicator and the error of the recovered gradient for the poisson problem with exact solution $u(x_1,x_2) = \arctan{\left(25(\sqrt{x_1^2+x_2^2} - 0.5)\right)}$ and various values of $h$, defined through $N+1$. For $r = 1$ on the left, $r = 3$ on the right and $N+1 = 6, 12, 24, 48, 96$ and $192$.}
\label{fig:Poisson:steep_gradient}
\end{figure}

\begin{figure}
    \includegraphics[width = \linewidth]{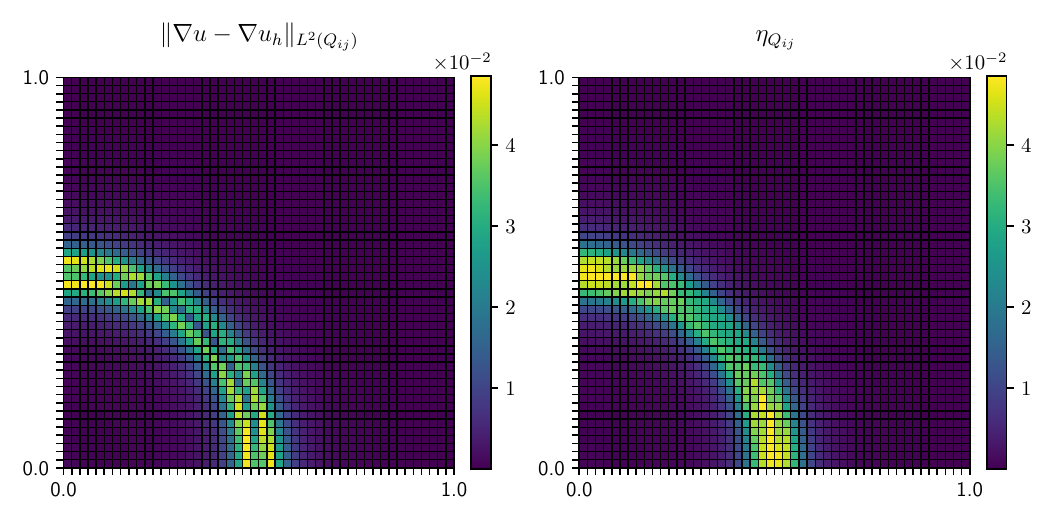}
    \caption{Local results for $r = 1$ and $h = \frac{1}{48}$ for the  poisson problem with exact solution $u(x_1,x_2) = \arctan{\left(25(\sqrt{x_1^2+x_2^2} - 0.5)\right)}$. On the left is the local error, and on the right is the local indicator. Additionally, we used the same color scale in both plots.}
    \label{fig:Poisson_const_coeff:steep_gradient:loc_err_k_2}
\end{figure}

\begin{figure}
    \includegraphics[width = \linewidth]{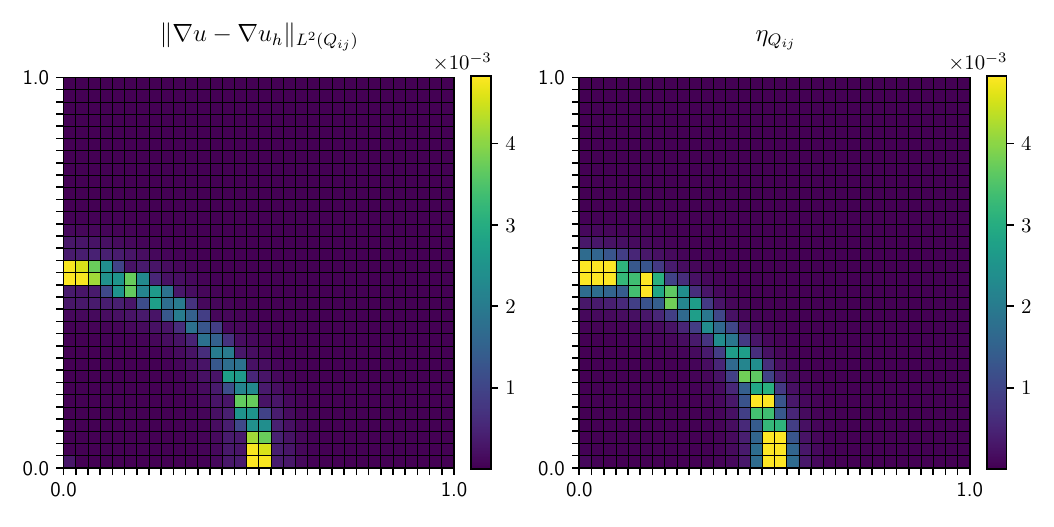}
    \caption{Local results for $r = 3$ and $h = \frac{1}{32}$ for the  poisson problem with exact solution $u(x_1,x_2) = \arctan{\left(25(\sqrt{x_1^2+x_2^2} - 0.5)\right)}$. On the left is the local error, and on the right is the local indicator. Additionally, we used the same color scale in both plots.}
    \label{fig:Poisson_const_coeff:steep_gradient:loc_err_k_4}
\end{figure}

\paragraph{Numerical example 3}\label{par:poisson:numerical_example_3}
In this example, we restrict to the second-order FDM, i.e. $r = 1$, and consider the Poisson equation in 2D with varying discontinuous coefficients:
\begin{equation}
 \begin{alignedat}{2}
    -\nabla \cdot (a (x_1,x_2) \nabla u(x_1,x_2)) &= \sin{\pi x_1}\sin{\pi x_2}&\quad& (x_1,x_2)\in \Omega,\\
    u(x_1,x_2) &= 0&\quad&   (x_1,x_2)\in \partial \Omega,
\end{alignedat}
\end{equation}
where
\begin{equation}
    a(x_1,x_2) =  \left\{ \begin{array}{cc}
        0.5 &\quad  x_1 \leq 0.5,\quad x_2 \leq 0.5,\\
        10 &\quad  x_1 > 0.5,\quad x_2 \leq 0.5,\\
        10 &\quad  x_1 \leq 0.5,\quad x_2 > 0.5,\\
        0.5 &\quad  x_1 > 0.5,\quad x_2 > 0.5.\\
    \end{array} \right.
\end{equation}
we use the numerical solution on a finer mesh with mesh parameter $h = \frac{1}{256}$ as the reference solution. We need to adjust the recovered gradient \(\mathrm{G}_h u_h = \left(\mathrm{G}_h^{x_1} u_h \; \mathrm{G}_h^{x_2} u_h\right)^T\) to account for the discontinuity in the gradient. We apply the PPR operator $\mathrm{G}_h$ separately in each subdomain where $a(x_1,x_2)$ is constant, allowing for a discontinuity along the specified interface. In Table \ref{tab:Poisson_disc_coeff:glob_conv}, we list the values and rates of the error in the natural norm, alongside the indicator, the error of the recovered gradient, and the effectivity index. We observe that the gradient of the error converges with order one. Moreover, the effectivity index converges to $1$, so the indicator is asymptotically exact. This is because the recovered gradient is superconvergent. In Figure \ref{fig:Poisson_const_coeff:loc_err_k_2}, the gradient of the error and the indicator are shown locally. We observe that, locally, the features of the error are captured accurately by the error indicator.
\begin{table}
\footnotesize
\caption{The gradient of the error, the indicator, the error of the recovered gradient and the effectivity index for $r = 1$ and the  poisson problem with varying discontinuous coefficients for various values of $h$, defined through $N+1$}
\label{tab:Poisson_disc_coeff:glob_conv}
\begin{tabular}{lllllllll}
\hline\noalign{\smallskip}
$N+1$& $h$ & \multicolumn{2}{l}{$||\nabla u - \nabla u_h||_{L^2(\Omega)}$} & \multicolumn{2}{l}{$\eta_h$} & \multicolumn{2}{l}{$||\nabla u - \mathrm{G}_h u_h||_{L^2(\Omega)}$} & e.i. \\
& & Value & Rate & Value & Rate & Value & Rate &  \\ \hline
\noalign{\smallskip}\hline\noalign{\smallskip}
4 & 1/4 & 3.979e-02 &  & 2.239e-02 &  & 2.691e-02 &  & 0.563 \\
8 & 1/8 & 2.043e-02 & 0.96 & 1.624e-02 & 0.46 & 1.031e-02 & 1.38 & 0.795 \\
16& 1/16& 1.012e-02 & 1.01 & 9.315e-03 & 0.80 & 3.163e-03 & 1.70 & 0.920 \\
32& 1/32& 4.996e-03 & 1.02 & 4.895e-03 & 0.93 & 9.617e-04 & 1.72 & 0.980 \\
64& 1/64& 2.547e-03 & 0.97 & 2.490e-03 & 0.98 & 3.500e-04 & 1.46 & 0.978 \\
\noalign{\smallskip}\hline
\end{tabular}
\end{table}

\begin{figure}
    \includegraphics[width = \linewidth]{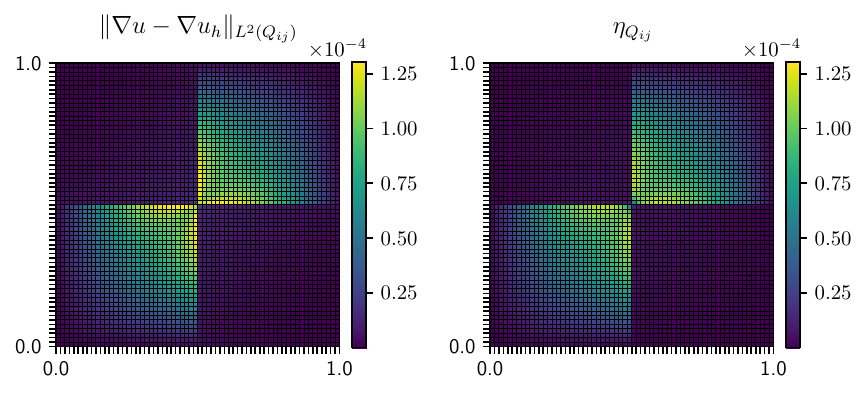}
    \caption{Local results for $r = 1$ and $h = \frac{1}{64}$ for the  poisson problem with varying discontinous coefficients. On the left is the local error, and on the right is the local indicator. Additionally, we used the same color scale in both plots.}
    \label{fig:Poisson_const_coeff:loc_err_k_2}
\end{figure}
\subsection{The Wave Equation}\label{subsection:Wave_equation}

In this section, we consider two numerical examples of the Wave equation. We first consider the 2D ($d = 3$) acoustic wave equation in a discontinuous heterogeneous medium. For this example, we restrict ourselves to the case $r = 1$ and examine the error in the non-natural norm and the RMS natural norm stated in equations \eqref{eq:a_priori_wave} and \eqref{eq:natural_norm_wave}, respectively. For the second example, we consider the 3D ($d = 4$) acoustic wave equation in a homogeneous medium. Specifically, we consider the cases $r = 1$ and $ r = 3$, and examine the error in the RMS natural norm. The FDMs used in this section can be found in \cite{cohen1996construction}. Moreover, we use a numerical solution on a fine grid as ground truth. We now introduce the following quantity to assess the convergence rate of the recovered gradient in the RMS natural norm:
\begin{equation*}
    \overline{|||\tilde{e}_{h}|||}^2 :=  \overline{\left\| \tilde{e}_{h}^t\right\|}^2_{L^2((0,1)^{d-1})} + \overline{\left\|\tilde{e}_{h}^x \right\|}^2_{L^2((0,1)^{d-1})}.
\end{equation*}
where 
\begin{align*}
 \overline{\left\| \tilde{e}_{h}^x\right\|}^2_{L^2((0,1)^{d-1})} &:= \overline{\|\nabla u - \mathrm{G}^x_{h}u_{h}\|}_{L^2((0,1)^{d-1})}^2 ,\\
 \overline{\left\| \tilde{e}_{h}^t\right\|}^2_{L^2((0,1)^{d-1})} &:= \overline{\left\|\frac{\partial u}{\partial t} - \mathrm{G}^t_{h}u_{h}\right\|}_{L^2((0,1)^{d-1})}^2. 
\end{align*}
and similarly to assess the convergence rate of the recovered gradient in the non-natural norm, denoted without a bar.

\paragraph{Numerical example 1}

In this example, we consider the 2D acoustic wave equation in a discontinuous, heterogeneous medium on the unit square, with a Gaussian pulse as initial condition:
\begin{equation*}
\begin{alignedat}{2}
    &\frac{\partial^2 u}{\partial t^2}(x_1,x_2,t)  
    - \nabla \cdot(\mu(x_1,x_2) \nabla u(x_1,x_2,t)) = 0,& \quad  &(x_1,x_2,t) \in \Omega,\\
     &u(x_1,x_2,t)= 0,& \quad &(x_1,x_2,t)\in \partial((0,1)^2)\times [0,T],\\
     &u(x_1,x_2,0) = f(x_1,x_2,t),& \quad &(x_1,x_2)\in (0,1)^2,\\
     &\frac{\partial u}{\partial t}(x_1,x_2,0) = 0,& \quad &(x_1,x_2)\in (0,1)^2,\\    
\end{alignedat}
\end{equation*}
where $\Omega = (0,1)^2\times(0,T)$ (with $d = 3$), $f(x_1,x_2,t) = \exp{-\frac{(x_1-0.25)^2 + (x_2-0.5)^2}{0.005}}$,  
\begin{equation*}
    \mu(x_1,x_2) = \left\{\begin{array}{cc}
         1 &x_1 \leq 0.5,  \\
         0.5 &x_1 > 0.5, 
    \end{array}\right.
\end{equation*} 
is the bulk modulus, and the wave speed is given by 
\begin{equation*}
c(x_1,x_2) = \sqrt{\frac{\mu(x_1,x_2)}{\rho(x_1,x_2)}} =  \left\{\begin{array}{cc}
          1 &x_1 \leq 0.5,  \\
         \frac{1}{\sqrt{2}} &x_1 > 0.5.
    \end{array}\right.
\end{equation*} 
In this numerical example, the heterogeneity in the medium reduces the local spatial wavelength, and combined with reflections at the boundaries, it generates very fine features that are difficult to resolve accurately. This makes it an interesting example to assess the performance of our indicator. In the following simulations, we have a CFL number of $0.705$ \iffalse $\frac{T(N_x+1)}{N_t+1}$ \fi and final time $T = 0.705$. At the final time, the Gaussian pulse has crossed the heterogeneous medium and reflected both at the material interface and the boundary on the left, producing wave interactions. Note that we have to adjust the recovered gradient $\mathrm{G}_{h}u_{h} = \begin{pmatrix} \mathrm{G}_{h}^{x_1}u_{h}& \mathrm{G}_{h}^{x_2}u_{h}& \mathrm{G}_{h}^{t}u_{h}\end{pmatrix}^T$ to capture the discontinuity in the spatial derivatives. We apply the PPR operator $\mathrm{G}_h$ separately in each subdomain where $\mu$ is constant. On the left in Figure \ref{fig:wave_eq_2D_var:glob_conv_T_1.41}, we plot the error in the RMS natural norm (see equation \eqref{eq:natural_norm_wave}), its indicator and the error of the recovered gradient. The effectivity index converges to 1, confirming that the indicator is asymptotically exact, since the recovered gradient is superconvergent. In Tables \ref{tab:wave_eq_2D_var:Gaussian:glob_conv_T_1.41_space} and \ref{tab:wave_eq_2D_var:Gaussian:glob_conv_T_1.41_time} we show the results for spatial and temporal derivatives of the error, respectively. In space, the effectivity index approaches 1 already on coarse meshes, while in time, finer resolutions are required. Local error features are accurately captured by the indicators, as illustrated in Figures \ref{fig:wave_eq_2D_var:Gaussian:loc_error_T_1.41_space} and \ref{fig:wave_eq_2D_var:Gaussian:loc_error_T_1.41_time}. Equivalent results for the non-natural norm (see equation \eqref{eq:a_priori_wave}) are provided on the right in Figure  \ref{fig:wave_eq_2D_var:glob_conv_T_1.41} and Tables \ref{tab:wave_eq_2D_var:Gaussian:glob_conv_T_1.41_space_non_nat} and \ref{tab:wave_eq_2D_var:Gaussian:glob_conv_T_1.41_time_non_nat}, where similar convergence is observed. Compared to the effectivity index for the error in the RMS natural norm, the effectivity index approaches 1 for coarser meshes.

% \begin{table}
% \caption{The error in the RMS natural norm, the indicator, the error of the recovered gradient, and the effectivity index for $r = 1$ and the wave equation in a discontinuous, heterogeneous medium for various values of $h$ (corresponding to different $N_x+1$ and $N_t+1$)}
% \label{tab:wave_eq_2D_var:Gaussian:glob_conv_T_1.41_nat_norm}
% \begin{tabular}{llllllllll}
% \hline\noalign{\smallskip}
% $N_t+1$ & $N_x+1$& $h$ & \multicolumn{2}{l}{$\overline{||| e_{h} |||} $} & \multicolumn{2}{l}{$\overline{\eta_{h}}$} & \multicolumn{2}{l}{$\overline{||| \tilde{e}_{h}|||} $} & e.i. \\
% & & & Value & Rate & Value & Rate & Value & Rate  \\ \hline
% \noalign{\smallskip}\hline\noalign{\smallskip}
% 32 &32 & 1/32 & 9.851e-01 &  & 4.836e-01 &  & 8.839e-01 &  & 0.491 \\
% 64 &64 & 1/64 & 3.671e-01 & 1.42 & 2.405e-01 & 1.01 & 2.983e-01 & 1.57 & 0.655 \\
% 128 &128 & 1/128 & 1.399e-01 & 1.39 & 1.180e-01 & 1.03 & 8.268e-02 & 1.85 & 0.843 \\
% 256 &256 & 1/256 & 6.313e-02 & 1.15 & 5.869e-02 & 1.01 & 2.550e-02 & 1.70 & 0.930 \\
% 512 &512 & 1/512 & 3.049e-02 & 1.05 & 2.931e-02 & 1.00 & 1.016e-02 & 1.33 & 0.961 \\
% \noalign{\smallskip}\hline
% \end{tabular}
% \end{table}

\begin{figure}
  \begin{subfigure}[b]{0.49\textwidth}
    \includegraphics[width=\linewidth]{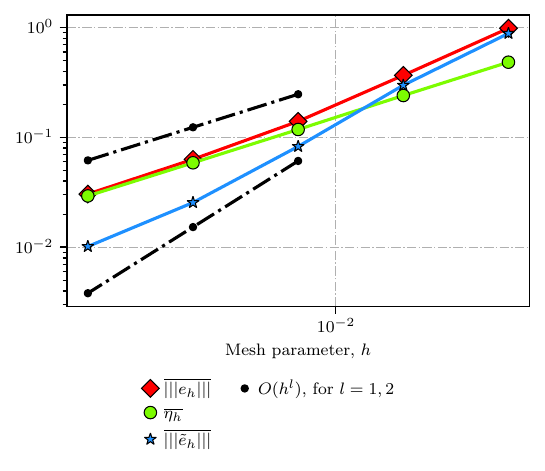}
  \end{subfigure}
  \hfill
  \begin{subfigure}[b]{0.49\textwidth}
    \includegraphics[width=\textwidth]{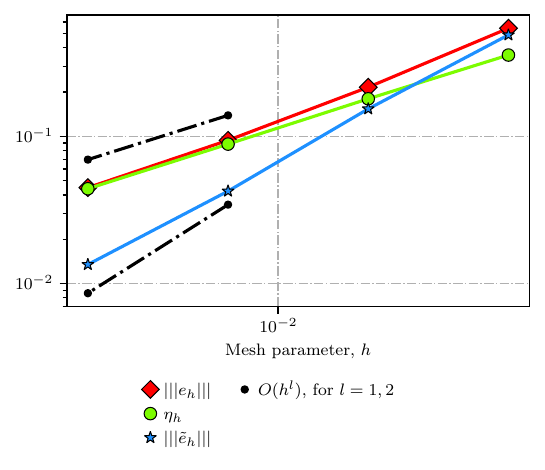}
  \end{subfigure}
\caption{Errors in the RMS natural norm (left) and non-natural norm (right), together with the indicator and recovered gradient error for the wave equation in a discontinuous, heterogeneous medium. Results are shown for \(r = 1\) and \(N_x+1 = N_t+1 = 32, 64, 128, 256\), and \(512\) (one less resolution level for the right plot).}
\label{fig:wave_eq_2D_var:glob_conv_T_1.41}
\end{figure}

\begin{table}
\footnotesize
\caption{The gradient of the error in the RMS norm, the spatial indicator, the error of the recovered gradient, and the effectivity index for $r = 1$, the wave equation in a discontinuous, heterogeneous medium and various values of $h$ (corresponding to different $N_x+1$ and $N_t+1$)}
\label{tab:wave_eq_2D_var:Gaussian:glob_conv_T_1.41_space}
\begin{tabular}{llllllllll}
\hline\noalign{\smallskip}
$N_t+1$ & $N_x+1$& $h$ & \multicolumn{2}{l}{$\overline{\left \| \nabla e_{h}\right \|}_{L^2((0,1)^{2})} $} & \multicolumn{2}{l}{$\overline{\eta_{h}^x}$} & \multicolumn{2}{l}{$\overline{|| \tilde{e}_{h}^x||}_{L^2((0,1)^{2})} $} & e.i. \\
& & & Value & Rate & Value & Rate & Value & Rate  \\ \hline
\noalign{\smallskip}\hline\noalign{\smallskip}
32& 32 & 1/32 & 7.937e-01 &  & 4.450e-01 &  & 6.776e-01 &  & 0.561 \\
64& 64 & 1/64 & 3.030e-01 & 1.39 & 2.148e-01 & 1.05 & 2.391e-01 & 1.50 & 0.709 \\
128& 128 & 1/128 & 1.188e-01 & 1.35 & 1.039e-01 & 1.05 & 6.716e-02 & 1.83 & 0.875 \\
256& 256 & 1/256 & 5.443e-02 & 1.13 & 5.146e-02 & 1.01 & 2.053e-02 & 1.71 & 0.945 \\
512& 512 & 1/512 & 2.643e-02 & 1.04 & 2.567e-02 & 1.00 & 8.140e-03 & 1.33 & 0.971 \\
\noalign{\smallskip}\hline
\end{tabular}
\end{table}

\begin{table}
\footnotesize
\caption{The temporal derivative of the error in the RMS norm, the temporal indicator, the error of the recovered temporal derivative, and the effectivity index for $r = 1$, the wave equation in a discontinuous, heterogeneous medium and various values of $h$ (corresponding to different $N_x+1$ and $N_t+1$)}
\label{tab:wave_eq_2D_var:Gaussian:glob_conv_T_1.41_time}
\begin{tabular}{llllllllll}
\hline\noalign{\smallskip}
$N_t+1$ & $N_x+1$ & $h$ & \multicolumn{2}{l}{$\overline{\left \| \frac{\partial e_{h}}{\partial t}\right \|}_{L^2((0,1)^{2})} $} & \multicolumn{2}{l}{$\overline{\eta_{h}^t}$} & \multicolumn{2}{l}{$\overline{||\tilde{e}_{h}^t||}_{L^2((0,1)^{2})} $} & e.i. \\
& & & Value & Rate & Value & Rate & Value & Rate  \\ \hline
\noalign{\smallskip}\hline\noalign{\smallskip}
32& 32& 1/32 & 5.835e-01 &  & 1.894e-01 &  & 5.677e-01 &  & 0.325 \\
64& 64& 1/64 & 2.072e-01 & 1.49 & 1.080e-01 & 0.81 & 1.783e-01 & 1.67 & 0.521 \\
128& 128& 1/128 & 7.394e-02 & 1.49 & 5.594e-02 & 0.95 & 4.822e-02 & 1.89 & 0.757 \\
256& 256& 1/256 & 3.198e-02 & 1.21 & 2.824e-02 & 0.99 & 1.513e-02 & 1.67 & 0.883 \\
512& 512& 1/512 & 1.521e-02 & 1.07 & 1.415e-02 & 1.00 & 6.077e-03 & 1.32 & 0.931 \\
\noalign{\smallskip}\hline
\end{tabular}
\end{table}
\begin{figure}
    \includegraphics[width = \linewidth]{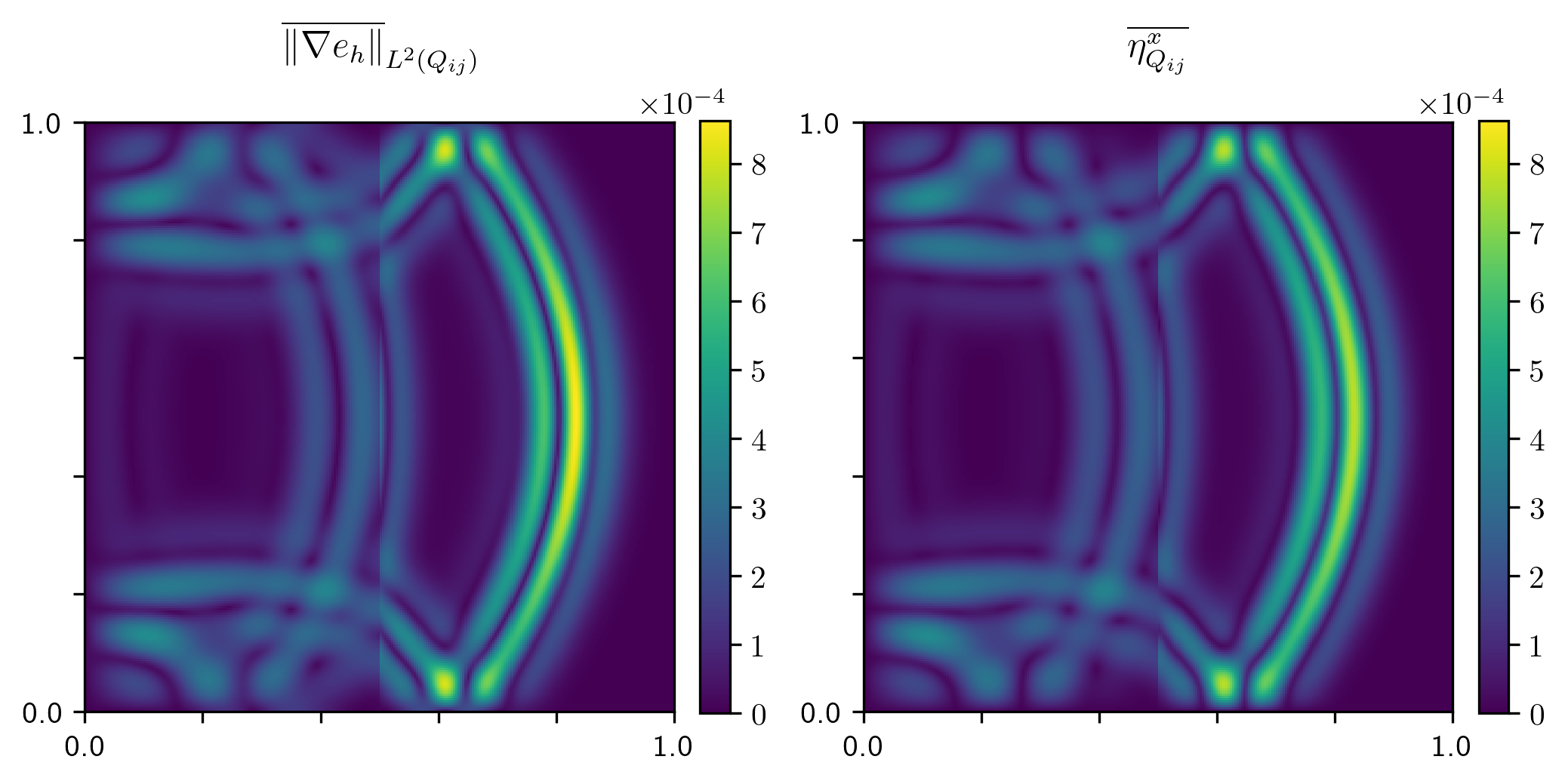}
    \caption{Local results for $r = 1,\;h = \frac{1}{256}$ and the acoustic wave equation in a discontinuous, heterogeneous medium. On the left is the local RMS gradient of the error, and on the right is its corresponding local indicator. Additionally, we used the same color scale in both plots.}
    \label{fig:wave_eq_2D_var:Gaussian:loc_error_T_1.41_space}
\end{figure}
\begin{figure}
    \includegraphics[width = \linewidth]{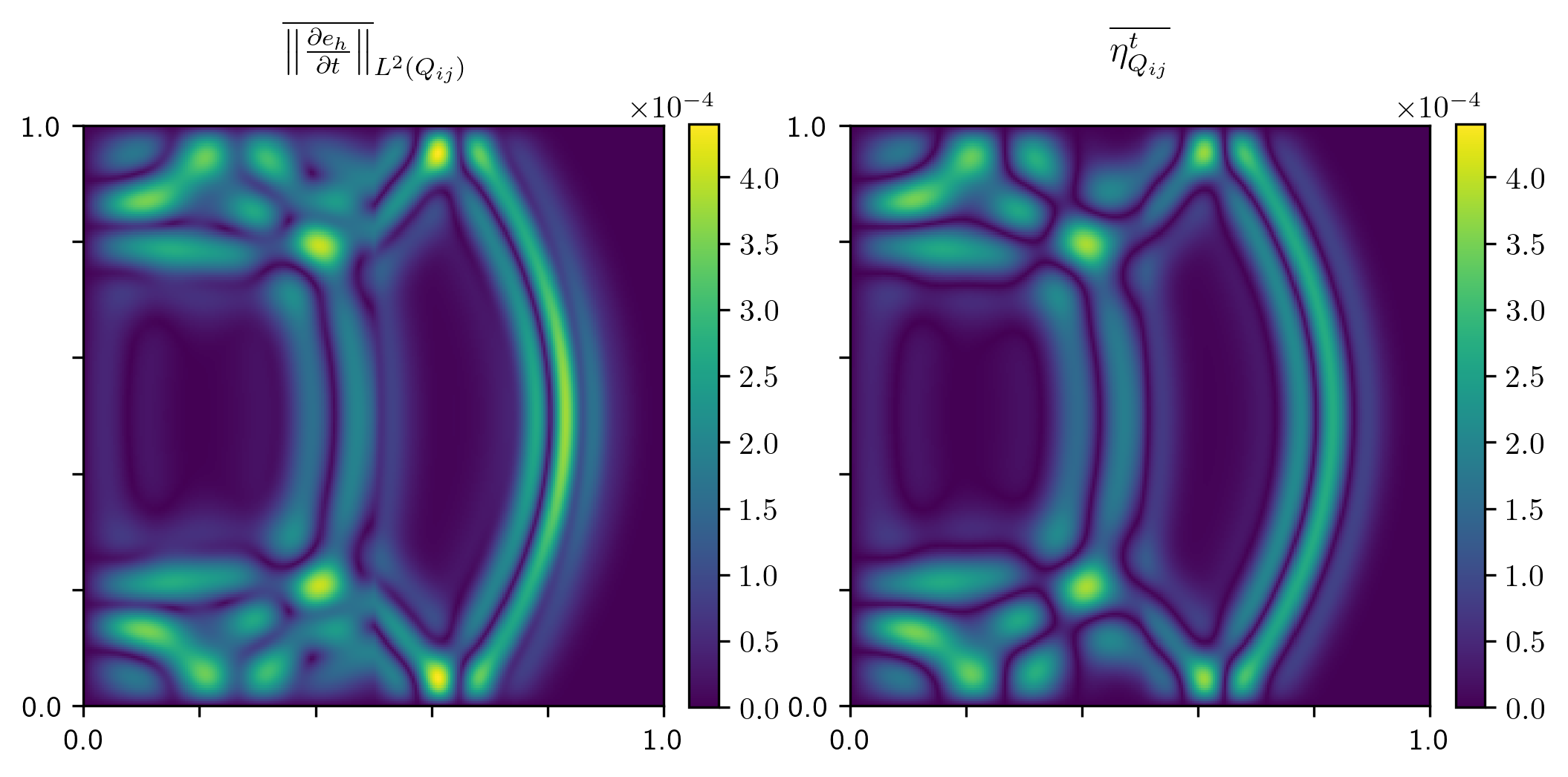}
    \caption{Local results for $r = 1,\;h = \frac{1}{256}$ and the acoustic wave equation in a discontinuous, heterogeneous medium. On the left is the local RMS temporal derivative of the error, and on the right is its corresponding local indicator. Additionally, we used the same color scale in both plots.}
    \label{fig:wave_eq_2D_var:Gaussian:loc_error_T_1.41_time}
\end{figure}

% \begin{table}
% \caption{The error in the non natural norm, the indicator, the error of the recovered gradient, and the effectivity index for $r = 1$ and the wave equation in a discontinuous, heterogeneous medium for various values of $h$ (corresponding to different $N_x+1$ and $N_t+1$)}
% \label{tab:wave_eq_2D_var:Gaussian:glob_conv_T_1.41_non_nat_norm}
% \begin{tabular}{llllllllll}
% \hline\noalign{\smallskip}
% $N_t+1$ &$N_x+1$ &$h$ & \multicolumn{2}{l}{$ \left \| \nabla e_{h}\right \|_{L^2(\Omega)} $} & \multicolumn{2}{l}{$\eta_{h}^x$} & \multicolumn{2}{l}{$||| \tilde{e}_{h}^x|||_{L^2(\Omega)}$} & e.i. \\
% & & & Value & Rate & Value & Rate & Value & Rate  \\ \hline
% \noalign{\smallskip}\hline\noalign{\smallskip}
% 32 & 32 & 1/32 & 4.333e-01 &  & 3.114e-01 &  & 3.809e-01 &  & 0.719 \\
% 64 & 64 & 1/64 & 1.767e-01 & 1.29 & 1.540e-01 & 1.02 & 1.224e-01 & 1.64 & 0.871 \\
% 128 & 128 & 1/128 & 7.832e-02 & 1.17 & 7.495e-02 & 1.04 & 3.398e-02 & 1.85 & 0.957 \\
% 256 & 256 & 1/256 & 3.758e-02 & 1.06 & 3.715e-02 & 1.01 & 1.081e-02 & 1.65 & 0.988 \\
% \noalign{\smallskip}\hline
% \end{tabular}
% \end{table}

\begin{table}
\footnotesize
\caption{The gradient of the error in the non natural norm, the spatial indicator, the error of the recovered gradient, and the effectivity index for $r = 1$, the wave equation in a discontinuous, heterogeneous medium and various values of $h$ (corresponding to different $N_x+1$ and $N_t+1$)}
\label{tab:wave_eq_2D_var:Gaussian:glob_conv_T_1.41_space_non_nat}
\begin{tabular}{llllllllll}
\hline\noalign{\smallskip}
$N_t+1$ &$N_x+1$ &$h$ & \multicolumn{2}{l}{$\left \| \frac{\partial e_{h}}{\partial t}\right \|_{L^2(\Omega)}$} & \multicolumn{2}{l}{$\eta_{h}^t$} & \multicolumn{2}{l}{$||| \tilde{e}_{h}^t|||_{L^2(\Omega)}$} & e.i. \\  
& & & Value & Rate & Value & Rate & Value & Rate  \\ \hline
\noalign{\smallskip}\hline\noalign{\smallskip}
32&32&1/32 & 3.277e-01 &  & 1.748e-01 &  & 3.069e-01 &  & 0.533 \\
64&64&1/64 & 1.240e-01 & 1.40 & 9.327e-02 & 0.91 & 9.305e-02 & 1.72 & 0.752 \\
128&128&1/128 & 5.201e-02 & 1.25 & 4.718e-02 & 0.98 & 2.539e-02 & 1.87 & 0.907 \\
256&256&1/256 & 2.441e-02 & 1.09 & 2.365e-02 & 1.00 & 7.993e-03 & 1.67 & 0.969 \\
\noalign{\smallskip}\hline
\end{tabular}
\end{table}

\begin{table}
\footnotesize
\caption{The temporal derivative of the error in the non-natural norm, the temporal indicator, the error of the recovered temporal derivative, and the effectivity index for $r = 1$, the wave equation in a discontinuous, heterogeneous medium and various values of $h$ (corresponding to different $N_x+1$ and $N_t+1$)}
\label{tab:wave_eq_2D_var:Gaussian:glob_conv_T_1.41_time_non_nat}
\begin{tabular}{llllllllll}
\hline\noalign{\smallskip}
$N_t+1$ &$N_x+1$ & $h$ & \multicolumn{2}{l}{$||| e_{h}||| $} & \multicolumn{2}{l}{$\eta_{h}$} & \multicolumn{2}{l}{$||| \tilde{e}_{h}||| $} & e.i. \\
& & & Value & Rate & Value & Rate & Value & Rate  \\ \hline
\noalign{\smallskip}\hline\noalign{\smallskip}
32&32&1/32 & 5.432e-01 &  & 3.571e-01 &  & 4.892e-01 &  & 0.657 \\
64&64&1/64 & 2.159e-01 & 1.33 & 1.800e-01 & 0.99 & 1.537e-01 & 1.67 & 0.834 \\
128&128&1/128 & 9.401e-02 & 1.20 & 8.857e-02 & 1.02 & 4.242e-02 & 1.86 & 0.942 \\
256&256&1/256 & 4.481e-02 & 1.07 & 4.404e-02 & 1.01 & 1.344e-02 & 1.66 & 0.983 \\
\noalign{\smallskip}\hline
\end{tabular}
\end{table}

\paragraph{Numerical example 2}

In this example, we consider the acoustic wave equation in a homogeneous medium with unit wave speed on the unit cube with a Gaussian pulse as an initial condition: 

\begin{equation*}
\begin{alignedat}{2}
     &\frac{\partial^2 u}{\partial t^2}(x_1,x_2,x_3,t) - \Delta u(x_1,x_2,x_3,t) = 0,&\quad  &(x_1,x_2,x_3,t) \in \Omega,\\
     &u(x_1,x_2,x_3,t)= 0,&\quad &(x_1,x_2,x_3,t)\in \partial((0,1)^3)\times [0,T],\\
     &u(x_1,x_2,x_3,0)= f(x_1,x_2,x_3,t),&\quad &(x_1,x_2,x_3)\in (0,1)^3,\\
     &\frac{\partial u}{\partial t}(x_1,x_2,x_3,0)= 0,&\quad & (x_1,x_2,x_3)\in (0,1)^3.\\    
\end{alignedat}
\end{equation*}
where $\Omega = (0,1)^3\times(0,T)$ (with $d = 4$) and $f(x_1,x_2,x_3,t) = \\ \exp{-\frac{(x_1-0.5)^2 + (x_2-0.5)^2 + (x_3-0.5)^2}{0.01}}$. Although the medium is homogeneous, reflections at the boundaries generate fine spatial features that are difficult to resolve accurately. We do not have an exact solution, and use a reference solution on a finer mesh as the ground truth. We set the CFL number in the following simulations to $0.5$ and the final time $T = 1$. At the final time, the Gaussian pulse has reflected off the boundaries and returned to its initial center position, producing many interacting waves. In Figure \ref{fig:wave_eq_3D_var:glob_conv_T_1}, we list the error in the RMS natural norm (see equation \eqref{eq:natural_norm_wave}), alongside the indicator and the error of the recovered gradient for $r = 1$ and $3$. We observe that the effectivity index converges to 1 for both $r = 1$ and $3$, and therefore the indicator is asymptotically exact. The recovered quantities in the RMS natural norm exhibit a superconvergent rate. In Tables \ref{tab:wave_eq_3D:Gaussian:glob_conv_T_0.5_space_r_1}, \ref{tab:wave_eq_3D:Gaussian:glob_conv_T_0.5_space_r_3}, \ref{tab:wave_eq_3D:Gaussian:glob_conv_T_0.5_time_r_1} and \ref{tab:wave_eq_3D:Gaussian:glob_conv_T_0.5_time_r_3} the spatial and temporal derivatives of the error are presented alongside their respective indicators, recovered gradients and effectivity indices for $r = 1$ and $3$. For the spatial derivatives of the error, the indicator demonstrates clear asymptotic exactness. In contrast, for the temporal derivatives, we observe an emerging trend suggesting that the indicator is approaching asymptotic exactness over refinement.
In Figures \ref{fig:wave_eq_3D:Gaussian:loc_error_k_2_space}, \ref{fig:wave_eq_3D:Gaussian:loc_error_k_2_time}, \ref{fig:wave_eq_3D:Gaussian:loc_error_k_4_space} and \ref{fig:wave_eq_3D:Gaussian:loc_error_k_4_time}, we have the local error due to space and time together with their respective indicator for $r = 1$ and $r = 3$, respectively. We observe that the error indicator due to space accurately captures the features of the error due to space, but in time, the error pattern is captured, but not the size of the features. We suggest considering the error in the RMS natural norm and not the spatial and temporal contributions separately.

\begin{figure}
  \begin{subfigure}[b]{0.49\textwidth}
    \includegraphics[width=\linewidth]{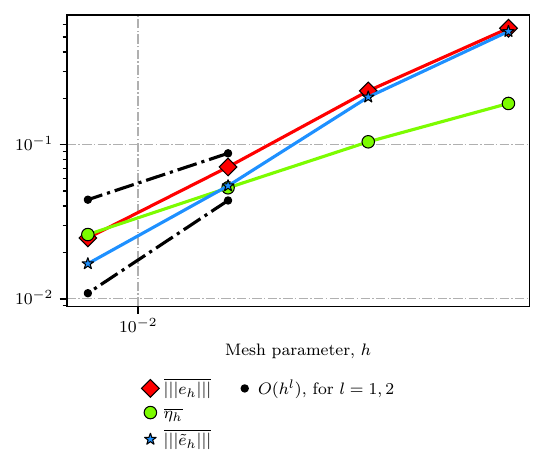}
  \end{subfigure}
  \hfill
  \begin{subfigure}[b]{0.49\textwidth}
    \includegraphics[width=\textwidth]{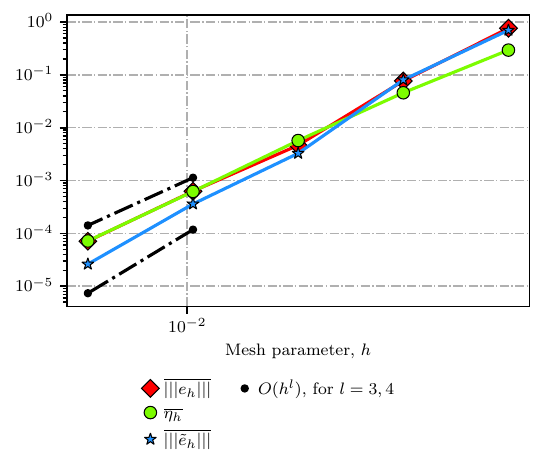}
  \end{subfigure}
\caption{Errors in the RMS natural norm, together with the indicator and recovered gradient error for the wave equation in a homogeneous medium. Results are shown for \(r = 1\) and \( (N_t+1, N_x+1) = (32, 16), (64,32), (128, 64)\), and \((256, 128) \) on the left and \(r = 3\) and \( (N_t+1, N_x+1) = (24, 12), (48,24), (96, 48), (192, 96)\), and \((384, 192) \) on the right.}
\label{fig:wave_eq_3D_var:glob_conv_T_1}
\end{figure}

\begin{table}
\footnotesize
\caption{The gradient of the error in the RMS norm, the spatial indicator, the error of the recovered gradient, and the effectivity index for $r = 1$, the wave equation in a homogeneous medium and various values of $h$ (corresponding to different $N_x+1$ and $N_t+1$)}
\label{tab:wave_eq_3D:Gaussian:glob_conv_T_0.5_space_r_1}
\begin{tabular}{llllllllll}
\hline\noalign{\smallskip}
$N_t+1$ & $N_x+1$ & $h$ & \multicolumn{2}{l}{$\overline{\left \| \nabla e_{h}\right \|}_{L^2((0,1)^{3})} $} & \multicolumn{2}{l}{$\overline{\eta_{h}^x}$} & \multicolumn{2}{l}{$\overline{||| \tilde{e}_{h}^x|||}_{L^2((0,1)^{3})} $} & e.i. \\
& & & Value & Rate & Value & Rate & Value & Rate  \\ \hline
\noalign{\smallskip}\hline\noalign{\smallskip}
32 &16 & 1/16 & 4.136e-01 &  & 1.686e-01 &  & 3.749e-01 &  & 0.408 \\
64 &32 & 1/32 & 1.587e-01 & 1.38 & 9.113e-02 & 0.89 & 1.379e-01 & 1.44 & 0.574 \\
128 &64 & 1/64 & 5.465e-02 & 1.54 & 4.473e-02 & 1.03 & 3.786e-02 & 1.87 & 0.818 \\
256 &128 & 1/128 & 2.028e-02 & 1.43 & 2.205e-02 & 1.02 & 1.339e-02 & 1.50 & 1.087 \\
\noalign{\smallskip}\hline
\end{tabular}
\end{table}

\begin{table}
\footnotesize
\caption{The gradient of the error in the RMS norm, the spatial indicator, the error of the recovered gradient, and the effectivity index for $r = 3$, the wave equation in a homogeneous medium and various values of $h$ (corresponding to different $N_x+1$ and $N_t+1$)}
\label{tab:wave_eq_3D:Gaussian:glob_conv_T_0.5_space_r_3}
\begin{tabular}{llllllllll}
\hline\noalign{\smallskip}
$N_t+1$ & $N_x+1$ & $h$ & \multicolumn{2}{l}{$\overline{\left \| \nabla e_{h}\right \|}_{L^2((0,1)^{3})} $} & \multicolumn{2}{l}{$\overline{\eta_{h}^x}$} & \multicolumn{2}{l}{$\overline{||| \tilde{e}_{h}^x|||}_{L^2((0,1)^{3})} $} & e.i. \\
& & & Value & Rate & Value & Rate & Value & Rate  \\ \hline
\noalign{\smallskip}\hline\noalign{\smallskip}
24& 12& 1/4 & 5.378e-01 &  & 2.685e-01 &  & 4.601e-01 &  & 0.499 \\
48& 24& 1/8 & 5.645e-02 & 3.25 & 4.411e-02 & 2.61 & 6.120e-02 & 2.91 & 0.781 \\
96& 48& 1/16 & 4.397e-03 & 3.68 & 5.581e-03 & 2.98 & 2.913e-03 & 4.39 & 1.269 \\
192& 96& 1/32 & 5.797e-04 & 2.92 & 6.065e-04 & 3.20 & 2.866e-04 & 3.35 & 1.046 \\
384& 192& 1/64 & 6.801e-05 & 3.09 & 7.012e-05 & 3.11 & 2.151e-05 & 3.74 & 1.031 \\
\noalign{\smallskip}\hline
\end{tabular}
\end{table}

\begin{table}
\footnotesize
\caption{The temporal derivative of the error in the RMS norm, the temporal indicator, the error of the recovered temporal derivative, and the effectivity index for $r = 1$, the wave equation in a homogeneous medium and various values of $h$ (corresponding to different $N_x+1$ and $N_t+1$)}
\label{tab:wave_eq_3D:Gaussian:glob_conv_T_0.5_time_r_1}
\begin{tabular}{llllllllll}
\hline\noalign{\smallskip}
$N_t+1$ & $N_x+1$ & $h$ & \multicolumn{2}{l}{$\overline{\left \| \frac{\partial e_{h}}{\partial t}\right \|}_{L^2((0,1)^{3})} $} & \multicolumn{2}{l}{$\overline{\eta_{h}^t}$} & \multicolumn{2}{l}{$\overline{||| \tilde{e}_{h}^t|||}_{L^2((0,1)^{3})} $} & e.i. \\
& & & Value & Rate & Value & Rate & Value & Rate  \\ \hline
\noalign{\smallskip}\hline\noalign{\smallskip}
32& 16 & 1/16 & 3.919e-01 &  & 7.633e-02 &  & 3.930e-01 &  & 0.195 \\
64& 32 & 1/32 & 1.573e-01 & 1.32 & 5.084e-02 & 0.59 & 1.500e-01 & 1.39 & 0.323 \\
128& 64 & 1/64 & 4.631e-02 & 1.76 & 2.744e-02 & 0.89 & 3.851e-02 & 1.96 & 0.592 \\
256& 128 & 1/128 & 1.424e-02 & 1.70 & 1.394e-02 & 0.98 & 1.026e-02 & 1.91 & 0.978 \\
\noalign{\smallskip}\hline
\end{tabular}
\end{table}

\begin{table}
\footnotesize
\caption{The temporal derivative of the error in the RMS norm, the temporal indicator, the error of the recovered temporal derivative, and the effectivity index for $r = 3$, the wave equation in a homogeneous medium and various values of $h$ (corresponding to different $N_x+1$ and $N_t+1$)}
\label{tab:wave_eq_3D:Gaussian:glob_conv_T_0.5_time_r_3}
\begin{tabular}{llllllllll}
\hline\noalign{\smallskip}
$N_t+1$ & $N_x+1$ & $h$ & \multicolumn{2}{l}{$\overline{\left \| \frac{\partial e_{h}}{\partial t}\right \|}_{L^2((0,1)^{3})} $} & \multicolumn{2}{l}{$\overline{\eta_{h}^t}$} & \multicolumn{2}{l}{$\overline{||| \tilde{e}_{h}^t|||}_{L^2((0,1)^{3})} $} & e.i. \\
& & & Value & Rate & Value & Rate & Value & Rate  \\ \hline
\noalign{\smallskip}\hline\noalign{\smallskip}
24& 12 & 1/4 & 5.410e-01 &  & 1.205e-01 &  & 5.280e-01 &  & 0.223 \\
48& 24 & 1/8 & 5.206e-02 & 3.38 & 1.237e-02 & 3.28 & 5.213e-02 & 3.34 & 0.238 \\
96& 48 & 1/16 & 1.504e-03 & 5.11 & 1.194e-03 & 3.37 & 1.491e-03 & 5.13 & 0.794 \\
192& 96 & 1/32 & 2.397e-04 & 2.65 & 1.279e-04 & 3.22 & 2.170e-04 & 2.78 & 0.534 \\
384& 192 & 1/64 & 2.018e-05 & 3.57 & 1.502e-05 & 3.09 & 1.484e-05 & 3.87 & 0.744 \\
\noalign{\smallskip}\hline
\end{tabular}
\end{table}

\begin{figure}
    \includegraphics[width = \linewidth]{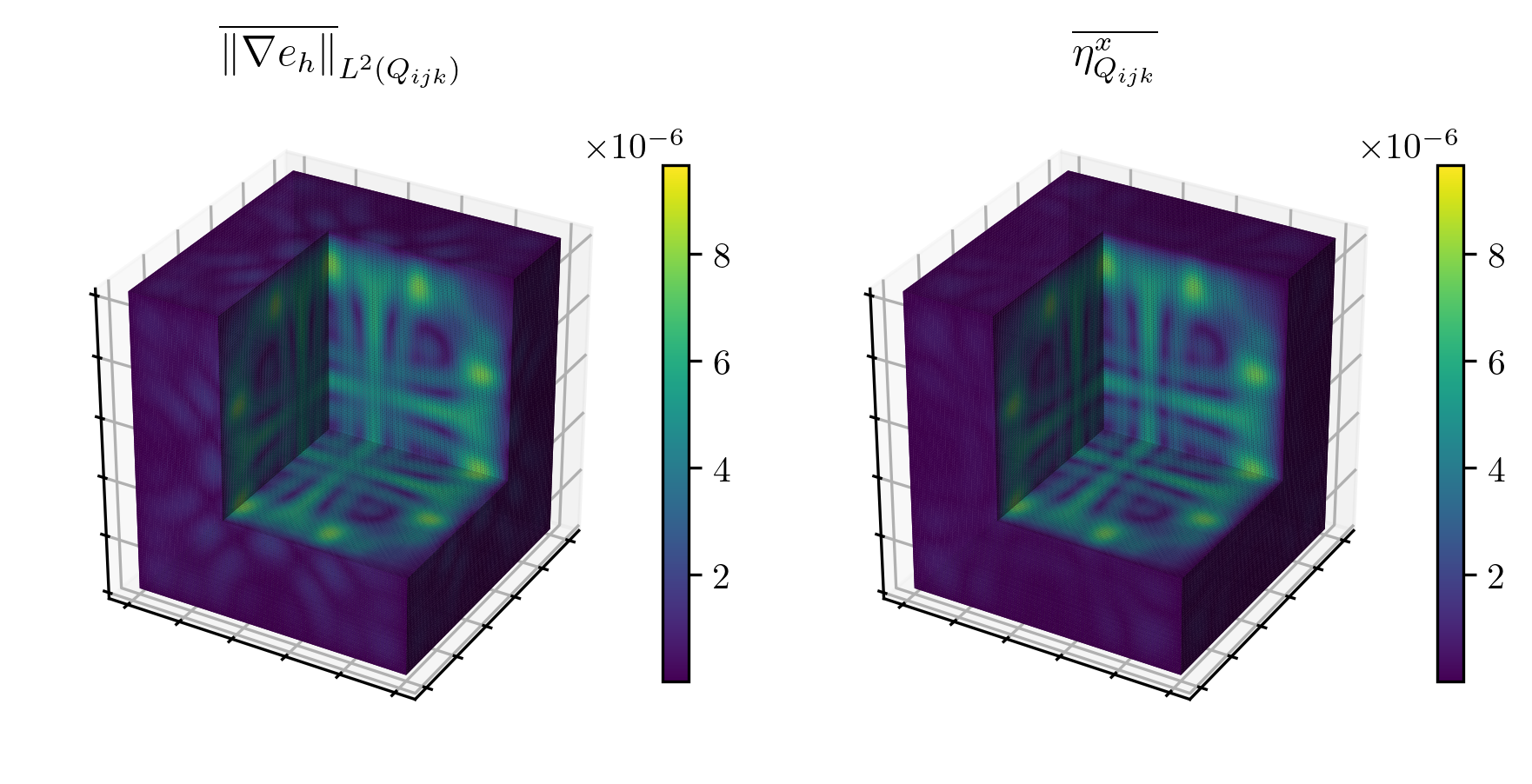}
    \caption{Local results for $r = 1,\;h = \frac{1}{128}$ and the acoustic wave equation in a homogeneous medium. On the left is the local RMS gradient of the error, and on the right is its corresponding local indicator. Additionally, we used the same color scale in both plots.}
    \label{fig:wave_eq_3D:Gaussian:loc_error_k_2_space}
\end{figure}
\begin{figure}
    \includegraphics[width = \linewidth]{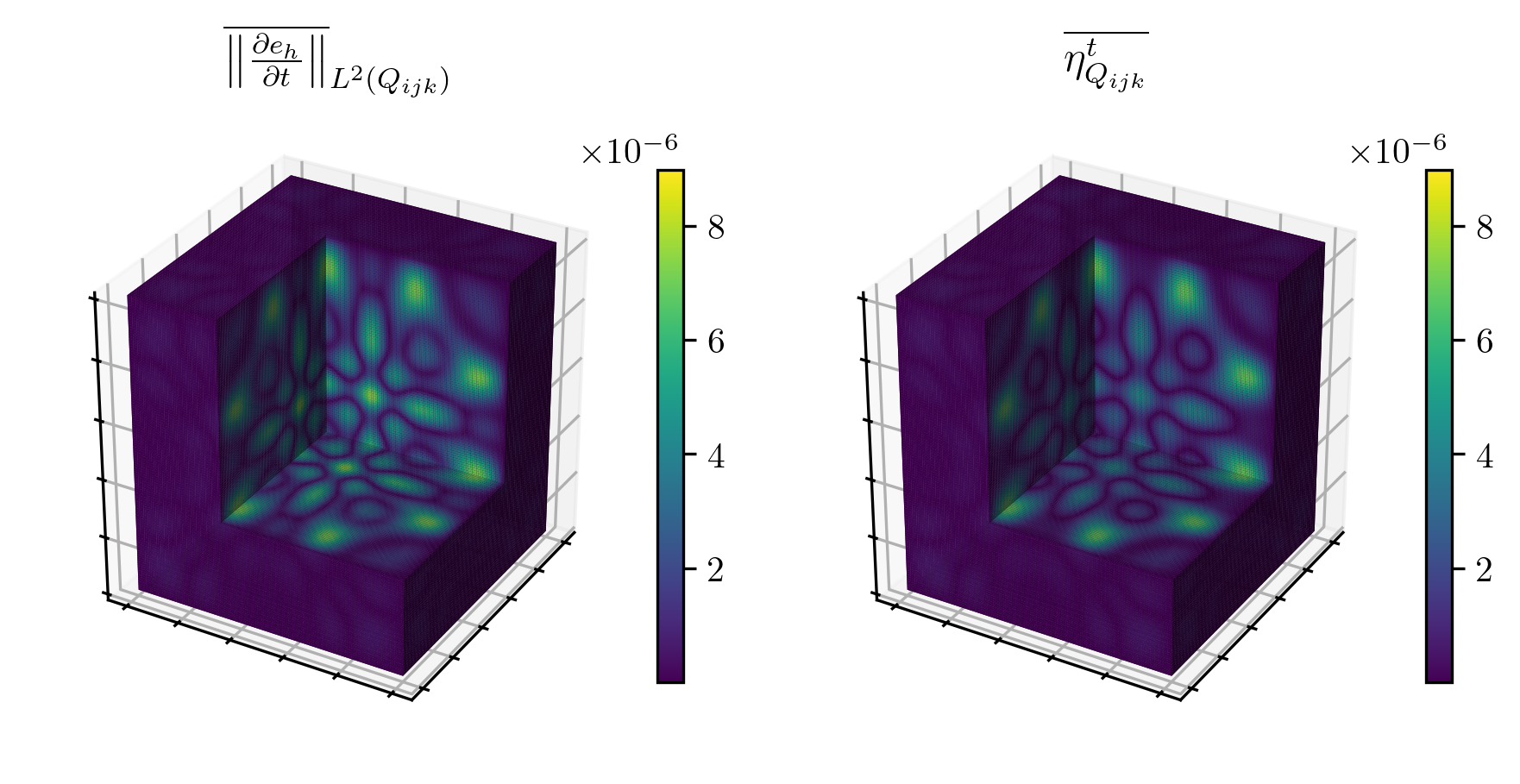}
    \caption{Local results for $r = 1,\;h = \frac{1}{128}$ and the acoustic wave equation in a homogeneous medium. On the left is the local RMS temporal derivative of the error, and on the right is its corresponding local indicator. Additionally, we used the same color scale in both plots.}
    \label{fig:wave_eq_3D:Gaussian:loc_error_k_2_time}
\end{figure}
\begin{figure}
    \includegraphics[width = \linewidth]{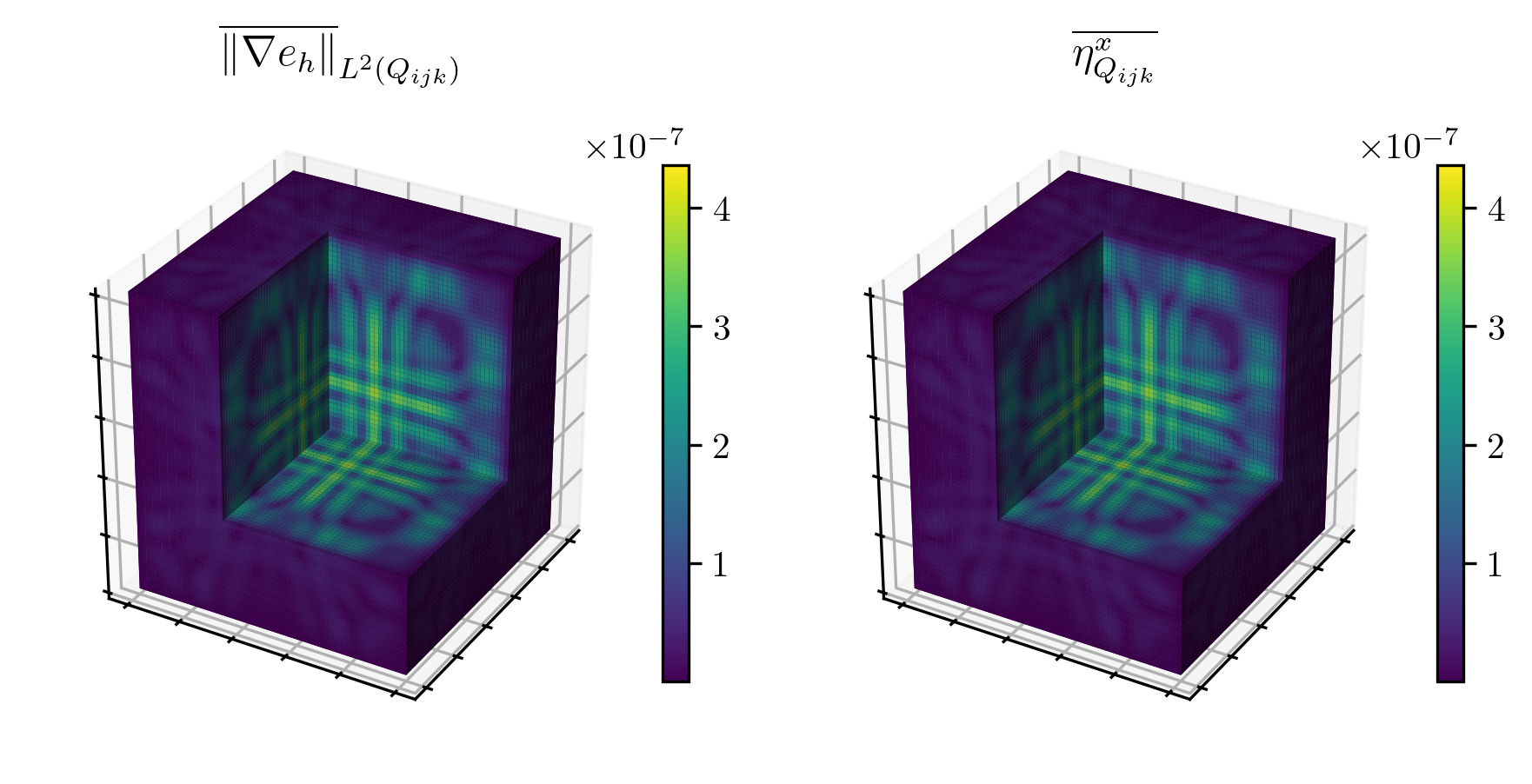}
    \caption{Local results for $r = 3,\;h = \frac{1}{64}$ and the acoustic wave equation in a homogeneous medium. On the left is the local RMS gradient of the error, and on the right is its corresponding local indicator. Additionally, we used the same color scale in both plots.}
    \label{fig:wave_eq_3D:Gaussian:loc_error_k_4_space}
\end{figure}
\begin{figure}
    \includegraphics[width = \linewidth]{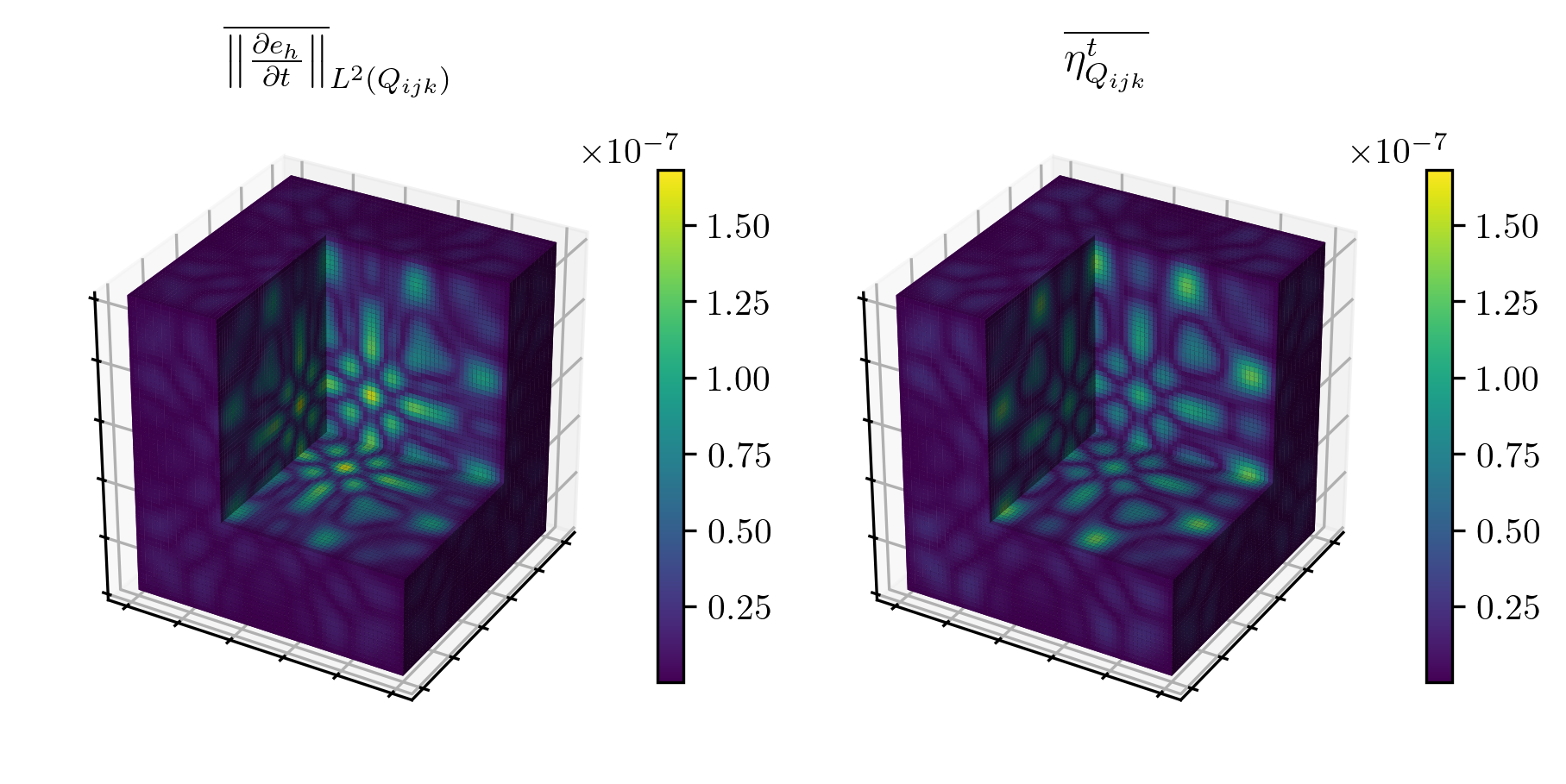}
    \caption{Local results for $r = 3,\;h = \frac{1}{64}$ and the acoustic wave equation in a homogeneous medium. On the left is the local RMS temporal derivative of the error, and on the right is its corresponding local indicator. Additionally, we used the same color scale in both plots.}
    \label{fig:wave_eq_3D:Gaussian:loc_error_k_4_time}
\end{figure}
%----------------------------------------------------------------------------------------------------------
%----------------------------------------------------------------------------------------------------------
%----------------------------------------------------------------------------------------------------------
\section{Conclusions}
\label{sec:conclusions}
We have proposed a novel recovery-based error indicator for the FDM, which could be applied as a black-box tool to a wide class of PDEs. By interpolating the FDM solution on a suitable mesh and using the recovery-based error estimator from Finite Element theory, we derive an error indicator for the gradient of the error and high-order FDM. Numerical experiments for the 2D Poisson problem—including discontinuous coefficients —demonstrate that the indicator is asymptotically exact in the gradient norm, for both second- and fourth-order FDMs. For the wave equation, we tested a 2D case with discontinuous coefficients (second-order FDM) and a 3D case with constant coefficients (second and fourth-order FDM). In these experiments, the indicator was shown to perform reliably in the RMS natural norm (see equation \eqref{eq:natural_norm_wave}). We also compared spatial and temporal error contributions separately, observing that temporal errors are more challenging to capture accurately, which motivates the use of the full RMS natural norm. Future research directions include exploring different interpolation strategies in type and order. 
%----------------------------------------------------------------------------------------------------------
%----------------------------------------------------------------------------------------------------------
%----------------------------------------------------------------------------------------------------------

\section*{Acknowledgments}
The authors would like to thank Paolo Ricci for valuable discussions. This work has been carried out within the framework of the EUROfusion Consortium, via the Euratom Research and Training Programme (Grant Agreement No 101052200 — EUROfusion) and funded by the Swiss State Secretariat for Education, Research and Innovation (SERI). Views and opinions expressed are however those of the author(s) only and do not necessarily reflect those of the European Union, the European Commission, or SERI. Neither the European Union nor the European Commission nor SERI can be held responsible for them.
\bibliographystyle{siamplain}
\bibliography{references}
\end{document}